\documentclass[pdfa,a4paper,UKenglish,cleveref,autoref,thm-restate]{lipics-v2021}  
\newcommand{\compilationdestination}{arxiv}

\makeatletter
\providecommand*{\toclevel@example}{0}
\providecommand*{\toclevel@theorem}{0}
\providecommand*{\toclevel@remark}{0}
\providecommand*{\toclevel@definition}{0}
\providecommand*{\toclevel@remark}{0}
\providecommand*{\toclevel@lemma}{0}
\makeatother
\usepackage{amsfonts,amsmath,amsthm,amssymb}
\DeclareMathSymbol{\sminus}{\mathbin}{AMSa}{"39}

\usepackage{anyfontsize}  
\usepackage{savesym}
\usepackage{amssymb}
\savesymbol{blacktriangleright}
\savesymbol{blacktriangleleft}
\savesymbol{vartriangleright}
\savesymbol{vartriangleleft}

\usepackage{mathabx}
\restoresymbol{TXF}{blacktriangleright} 
\restoresymbol{TXF}{blacktriangleleft} 
\restoresymbol{TXF}{vartriangleright} 
\restoresymbol{TXF}{vartriangleleft}

\usepackage{ifthen}

\usepackage{amsopn} 
\usepackage{tikz}
\usepackage{mathtools}
\usepackage{soul}
\usepackage{colortbl} \usepackage{makecell}  
\usepackage{tkz-euclide}
\usepackage{pgfplots}
\usetikzlibrary{patterns}
\def\scalesize{0.25}
\def\myraisebox{-4.5pt}
\def\arrowlength{4.5pt}
\def\shortarrowlength{4pt}

\def\King{\raisebox{\myraisebox}{\begin{tikzpicture}[scale = \scalesize]
 \node(O) at (0, 0){};
 \node(E) at (1, 0){};
 \node(SE) at (1, -1){};
 \node(S) at (0, -1){};
 \node(SW) at (-1, -1){};
 \node(W) at (-1, 0){};
 \node(NW) at (-1, 1){};
 \node(N) at (0, 1){};
 \node(NE) at (1, 1){};
 \draw[help lines] (-1, -1) grid (1, 1);
\draw [blue, fill] (0,0) circle (4pt);
\draw[blue, thick, line cap = round, ->, arrows={-Latex[length=\arrowlength]}] (O.center) to (NW.center);
\draw[blue, thick, line cap = round, ->, arrows={-Latex[length=\arrowlength]}] (O.center) to (NE.center);
\draw[blue, thick, line cap = round, ->, arrows={-Latex[length=\arrowlength]}] (O.center) to (SE.center);
\draw[blue, thick, line cap = round, ->, arrows={-Latex[length=\arrowlength]}] (O.center) to (SW.center);
\draw[blue, thick, line cap = round, ->, arrows={-Latex[length=\shortarrowlength]}] (O.center) to (E.center);
\draw[blue, thick, line cap = round, ->, arrows={-Latex[length=\shortarrowlength]}] (O.center) to (W.center);
\draw[blue, thick, line cap = round, ->, arrows={-Latex[length=\shortarrowlength]}] (O.center) to (N.center);
\draw[blue, thick, line cap = round, ->, arrows={-Latex[length=\shortarrowlength]}] (O.center) to (S.center);
\end{tikzpicture}}}

\def\Diagonal{\raisebox{\myraisebox}{\begin{tikzpicture}[scale = \scalesize]
 \node(O) at (0, 0){};
 \node(E) at (1, 0){};
 \node(SE) at (1, -1){};
 \node(S) at (0, -1){};
 \node(SW) at (-1, -1){};
 \node(W) at (-1, 0){};
 \node(NW) at (-1, 1){};
 \node(N) at (0, 1){};
 \node(NE) at (1, 1){};
 \draw[help lines] (-1, -1) grid (1, 1);
\draw[blue, thick, line cap = round, ->, arrows={-Latex[length=\arrowlength]}] (O.center) to (NW.center);
\draw[blue, thick, line cap = round, ->, arrows={-Latex[length=\arrowlength]}] (O.center) to (NE.center);
\draw[blue, thick, line cap = round, ->, arrows={-Latex[length=\arrowlength]}] (O.center) to (SE.center);
\draw[blue, thick, line cap = round, ->, arrows={-Latex[length=\arrowlength]}] (O.center) to (SW.center);
\end{tikzpicture}}}

\def\Diabolo{\raisebox{\myraisebox}{\begin{tikzpicture}[scale = \scalesize]
 \node(O) at (0, 0){};
 \node(E) at (1, 0){};
 \node(SE) at (1, -1){};
 \node(S) at (0, -1){};
 \node(SW) at (-1, -1){};
 \node(W) at (-1, 0){};
 \node(NW) at (-1, 1){};
 \node(N) at (0, 1){};
 \node(NE) at (1, 1){};
 \draw[help lines] (-1, -1) grid (1, 1);
\draw [blue, fill] (0,0) circle (4pt);
\draw[blue, thick, line cap = round, ->, arrows={-Latex[length=\arrowlength]}] (O.center) to (NW.center);
\draw[blue, thick, line cap = round, ->, arrows={-Latex[length=\arrowlength]}] (O.center) to (NE.center);
\draw[blue, thick, line cap = round, ->, arrows={-Latex[length=\arrowlength]}] (O.center) to (SE.center);
\draw[blue, thick, line cap = round, ->, arrows={-Latex[length=\arrowlength]}] (O.center) to (SW.center);
\draw[blue, thick, line cap = round, ->, arrows={-Latex[length=\shortarrowlength]}] (O.center) to (N.center);
\draw[blue, thick, line cap = round, ->, arrows={-Latex[length=\shortarrowlength]}] (O.center) to (S.center);
\end{tikzpicture}}}

\usepgfplotslibrary{fillbetween}
\pgfplotsset{compat=1.18}
\usepackage{graphicx} 

\usepackage{booktabs} \usepackage{cases}  
\usepackage{cite} 

\RequirePackage{xcolor}
\definecolor{darkgreen}{rgb}{0,0.4,0}
\definecolor{green}{rgb}{0,0.4,0}
\definecolor{BrickRed}{rgb}{0.65,0.08,0}
\RequirePackage{hyperref}
\ifthenelse{\equal{\compilationdestination}{arxiv}}{\hypersetup{colorlinks=true,linkcolor=blue,citecolor=darkgreen,filecolor=BrickRed,urlcolor=darkgreen}}{}
\newcommand{\OEISs}[1]{\href{http://oeis.org/#1}{#1}}
\newcommand{\N}{\ensuremath{\mathbb{N}}}
\newcommand{\R}{\ensuremath{\mathbb{R}}}
\newcommand{\Seq}{\text{\textsc{Seq}}}

\DeclareMathOperator{\nb}{\text{NegBin}}
\DeclareMathOperator{\ray}{\text{Rayleigh}}
\DeclareMathOperator{\Nc}{\mathcal{N}}
\newcommand{\claw}{\ensuremath{\xrightarrow{\mathcal{L}}}}

\def\P{{\mathbb {P}}}
\def\E{{\mathbb {E}}}
  
\newcommand{\cF}{c_F^{}}
\newcommand{\cG}{c_G^{}}
\newcommand{\cH}{c_H^{}}
\newcommand{\rhoF}{\rho_F^{}}
\newcommand{\rhoG}{\rho_G^{}}
\newcommand{\rhoH}{\rho_H^{}}
\newcommand{\tauH}{\tau_H^{}}
\newcommand{\lambdaF}{\lambda_F^{}}
\newcommand{\lambdaG}{\lambda_G^{}}
\newcommand{\lambdaH}{\lambda_H^{}}
\newcommand{\lambdaM}{\lambda_M^{}}
\newcommand{\Xnq}{\ensuremath{X_n(q)}}
\newcommand{\q}{\texorpdfstring{$q$}{q}}
\def\rhoq{\rho}

\title{Composition Schemes: \texorpdfstring{\boldmath $q$}{q}-Enumerations and Phase Transitions in Gibbs Models}

\author{Cyril Banderier}{Laboratoire d'Informatique de Paris Nord, Universit\'{e} Sorbonne Paris Nord, Villetaneuse, France \and \url{https://lipn.univ-paris13.fr/~banderier/}}{}{https://orcid.org/0000-0003-0755-3022}{supported by the French-Austrian PHC Amadeus project ``Asymptotic behaviour of combinatorial structures''.}
\author{Markus Kuba}{Department Applied Mathematics \& Physics, University of Applied Sciences - Technikum Wien, Austria \and \url{https://www.dmg.tuwien.ac.at/kuba/}}{}{https://orcid.org/0000-0001-7188-6601}{}
\author{Stephan Wagner}{Institute of Discrete Mathematics, TU Graz, Austria \and Department of Mathematics, Uppsala Universitet, Sweden \and \url{https://www.math.tugraz.at/~wagner/}}{}{https://orcid.org/0000-0001-5533-2764}{supported by the Swedish research council (VR), grant 2022-04030.}
\author{Michael Wallner}{Institut f\"{u}r Diskrete Mathematik und Geometrie, TU Wien, Austria \and \url{https://dmg.tuwien.ac.at/mwallner/}}{}{https://orcid.org/0000-0001-8581-449X}{supported by the Austrian Science Fund (FWF):~P~34142 and OeAD WTZ project FR~01/2023.}

\authorrunning{C. Banderier, M. Kuba, S. Wagner, and M. Wallner}
\Copyright{Cyril Banderier, Markus Kuba, Stephan Wagner, and Michael Wallner}



\relatedversiondetails{arXiv version}{https://arxiv.org/abs/2311.17226}  

\acknowledgements{
The second author warmly thanks Thomas Feierl for many discussions about watermelons and Paul Schreivogl for discussions on partition functions.
We also thank Christian Krattenthaler for drawing our attention to the open problem of the phase transition of contacts in watermelons,
and the two referees for their feedback.
Last but not least, we are pleased to dedicate this excursion into the realm of composition schemes and limit laws to one of the greatest experts of this realm,
namely Michael Drmota, at the occasion of his 60th birthday!}

\EventEditors{C\'{e}cile Mailler and Sebastian Wild}
\EventNoEds{2}
\EventLongTitle{35th International Conference on Probabilistic, Combinatorial and Asymptotic Methods for the Analysis of Algorithms (AofA 2024)}
\EventShortTitle{AofA 2024}
\EventAcronym{AofA}
\EventYear{2024}
\EventDate{June 17--21, 2024}
\EventLocation{University of Bath, UK}
\EventLogo{}
\SeriesVolume{302}
\ArticleNo{7}

\nolinenumbers

\begin{document}

\sloppy

\maketitle
\vspace{0.5\baselineskip}
\enlargethispage{-1.6\baselineskip}
\begin{abstract}
Composition schemes are ubiquitous in combinatorics, statistical mechanics and probability theory. 
We give a unifying explanation to various phenomena 
observed in the combinatorial and statistical physics literature in the context of~$q$-enumeration
(this is a model where objects with a parameter of value $k$ have a Gibbs measure/Boltzmann weight $q^k$).
For structures enumerated by a composition scheme, we prove a phase transition for any parameter having such a Gibbs measure: 
for a critical value $q=q_c$, the limit law of the parameter is a two-parameter Mittag-Leffler distribution,
while it is Gaussian in the supercritical regime ($q>q_c$),  
and it is a Boltzmann distribution in the subcritical regime ($0<q<q_c$).
We apply our results to fundamental statistics of lattice paths and quarter-plane walks. We also explain previously observed limit laws for pattern-restricted permutations, 
and a phenomenon uncovered by Krattenthaler for the wall contacts in watermelons.
\end{abstract}

\section{Introduction}
\subsection{\texorpdfstring{\boldmath $q$}{q}-enumeration and Gibbs distributions}
Let $\mathcal{T}$ be a family of combinatorial objects,  let $|\cdot|$ denote the size of objects,
and let $\mathcal{X}\colon \mathcal{T}\to\N$ be a statistic defined on~$\mathcal{T}$. The statistic $\mathcal{X}$ on objects of $\mathcal{T}$ of size $n$ can be encoded by the sum
\begin{equation}
\label{eqn:partFunc1}
f_n(q)=\sum_{T\in\mathcal{T}: |T|=n}q^{\mathcal{X}(T)}.
\end{equation}
This sum reduces for $q=1$ to the total number $f_n(1)=f_n$ of objects of size $n$. 
In combinatorics, for any given $q\in\R$,  it is called the \emph{$q$-enumeration} of $\mathcal{T}$ of size $n$ with respect to $\mathcal{X}$
(see, e.g., \cite{Aigner2021,HopkinsLazarLinusson2021}). In the language of statistical mechanics,  $f_n(q)$ is a partition function with Boltzmann weight $q$. 
The bivariate generating function $F(z,q)$ is then defined as
\begin{equation}
\label{eqn:gf1}
F(z,q)=\sum_{T\in\mathcal{T}}z^{|T|}q^{\mathcal{X}(T)}=\sum_{n\ge 0} f_n(q) z^n =\sum_{n\ge 0}\sum_{k\ge 0}f_{n,k}z^n q^k.
\end{equation}
Here $f_{n,k}$ denotes the number of objects of $\mathcal{T}$ of size $n$ for which $\mathcal{X}$ equals $k$. 
It is usual to associate with the statistic $\mathcal{X}$ the random variables $X_n$, $n\ge 1$, defined as
\begin{equation}
\label{eqn:defXn}
\P(X_n=k)=\frac{f_{n,k}}{f_n}=\frac{[q^k]f_n(q)}{f_n(1)}=\frac{[z^n q^k]F(z,q)}{[z^n]F(z,1)},
\end{equation}
such that each object from $\mathcal{T}$ of size $n$ is equally likely. 
The associated probability generating function is given by
$\E(q^{X_n})=\frac{[z^n]F(z,q)}{[z^n]F(z,1)}$. 
In Equation~\eqref{eqn:defXn}, the reader is probably used to consider $q$ as a formal variable, 
but in this work, like in statistical mechanics, we shall consider $q$ as an adjustable parameter (weight $\in\R_{+}$) of the underlying combinatorial and physical structures.
This is also the spirit of the Boltzmann sampling method~\cite{Boltz1}, 
where $q$ is tuned to minimize the number of rejection steps in the sampling algorithm.

More precisely, in this article, we put a Gibbs measure on the statistic $\mathcal{X}$; that is, one has the following probabilistic model.
\begin{definition}[Gibbs distribution]
\label{def:inducedRV}
Let a family $\mathcal{T}$ of combinatorial objects and a statistic $\mathcal{X}\colon \mathcal{T}\to\N$ be given.  
For real $q>0$, the Gibbs distribution of this statistic is the law of the random variable $\Xnq$ with probability mass function
\[
\P(\Xnq=k)=\frac{f_{n,k}q^k }{f_n(q)},\quad k\ge 0.
\]
In terms of the probability generating function $p(v)=\E(v^{X_n(1)})$, we have
$\E(v^{\Xnq})=\frac{p(vq)}{p(q)}$.\end{definition}
A well-known example in probability theory is the Mallows distribution~\cite{Mallows1957} on permutations with respect to the inversion statistic. 

In many applications, one is interested in the limit distribution of $X_n(q)$, which depends on the value of $q>0$; see, e.g., \cite{BEO1998,BGR2009,BEO2001,CiucuKrattenthaler2002,KrattenthalerGuttmannViennot2000,Krattenthaler2006,OP2019,TOR2014}.
Let us pinpoint the result of Krattenthaler~\cite{Krattenthaler2006}, who uncovered a phase transition in the normalized \emph{mean number} of wall contacts at $q=2$ in watermelons. 
Using methods from analytic combinatorics, we will show
that similar phase transitions naturally occur in a great many instances, not only with respect to the expectation but also for the \emph{limit laws}. 
We will use the framework of composition schemes, which often provide a direct and unifying way to explain why phase transitions occur~\cite{BFSS2001,BKW2021a,FS2009}. In Section~\ref{sec:Theorem}, we establish in which way the phase transitions 
in the Gibbs model depend on the value of~$q$. 
Particular instances of similar phenomena have been observed in~\cite{ChelikavadaPanzo2023,Wu2022}. We give further examples
in Sections~\ref{sec:Applications} and~\ref{sec:Extensions}.

\subsection{Composition schemes and Gibbs distributions}
Functional composition schemes such as 
$F(z)=G\big(H(z)\big)$ are of great importance in combinatorics~\cite{BFSS2001,BKW2021a,FS2009} and probability theory~\cite{Stufler2022}.
The main focus is to analyse probabilistic properties of compositions like
\begin{equation}
\label{Scheme:1}
F(z,u)=G\big(u H(z)\big)\end{equation}
as a multitude of parameters $\mathcal{X}$ can be modelled in this way. Here $u$ marks the so-called size of the core, i.e., the involved $G$-component; see~\cite{BFSS2001,BKW2021a,FS2009}. The distribution of the corresponding random variable $X_n$ is then readily defined by
\begin{equation*}
\P(X_n=k)= \frac{[z^n u^k]F(z,u)}{[z^n]F(z,1)}.
\end{equation*}
Structurally, such schemes are at the heart of many fascinating phase transition phenomena (analytically corresponding, e.g., to coalescing saddle points 
or to confluence of singularities), related to the Gibbs measure in statistical physics and probability theory~\cite{Stufler2022}. 

First, we relate composition schemes to $q$-enumeration.
\begin{lemma}[Composition schemes and Gibbs distributions]
\label{lem:basicProb}
Let a combinatorial statistic $\mathcal{X}$ with bivariate generating function $F(z,u)$ be given. 
Then, for real $q>0$ the Gibbs distribution of $\mathcal{X}$ has a probability mass function given in terms of $F$ by
\begin{equation*}
\P(\Xnq=k)= \frac{[z^n u^k]F(z,q u )}{[z^n]F(z,q)}.
\end{equation*}
\end{lemma}

\begin{figure}[htb!]
\newcommand{\myheight}{3.1cm}
\centering
\includegraphics[height=\myheight]{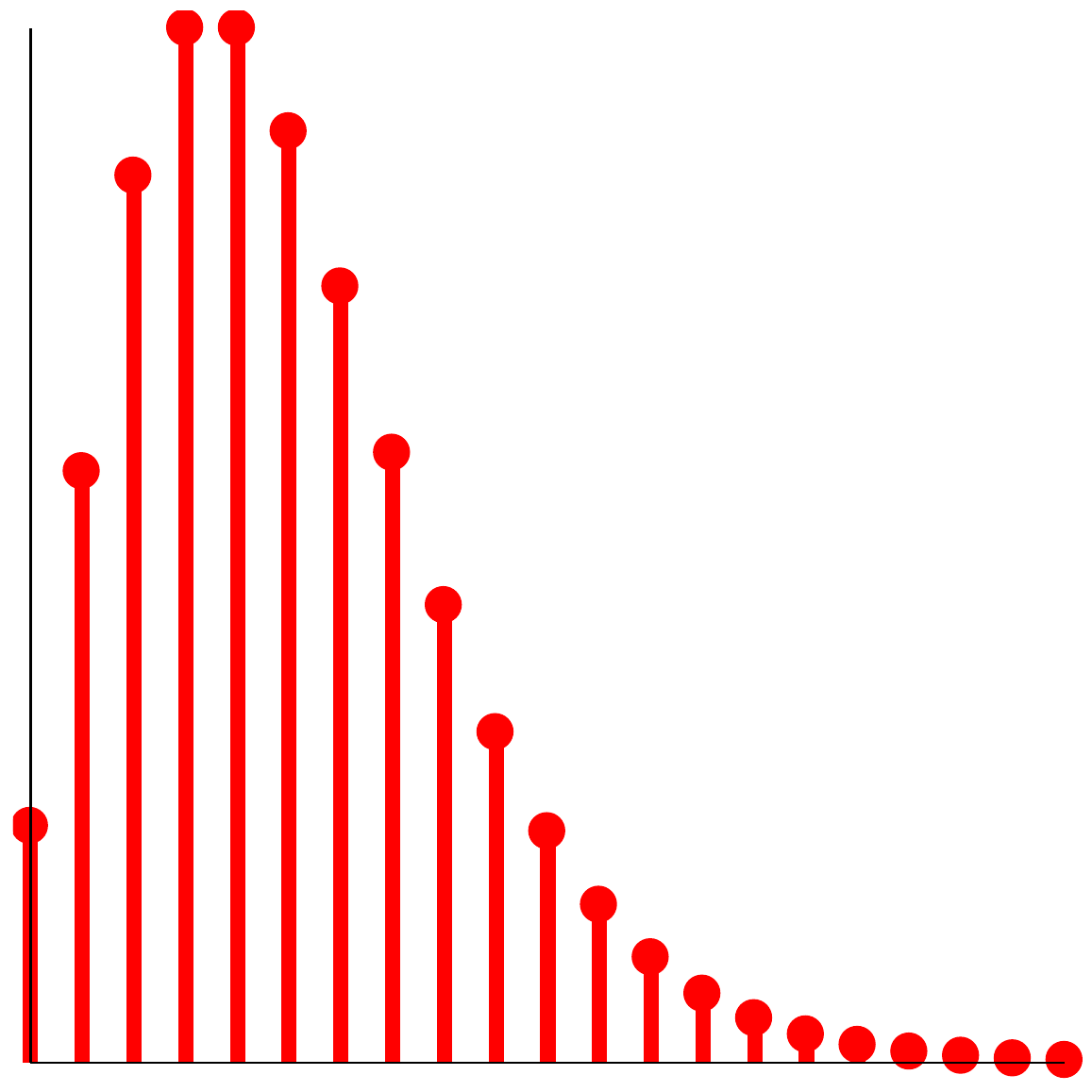}  \quad
\includegraphics[height=\myheight]{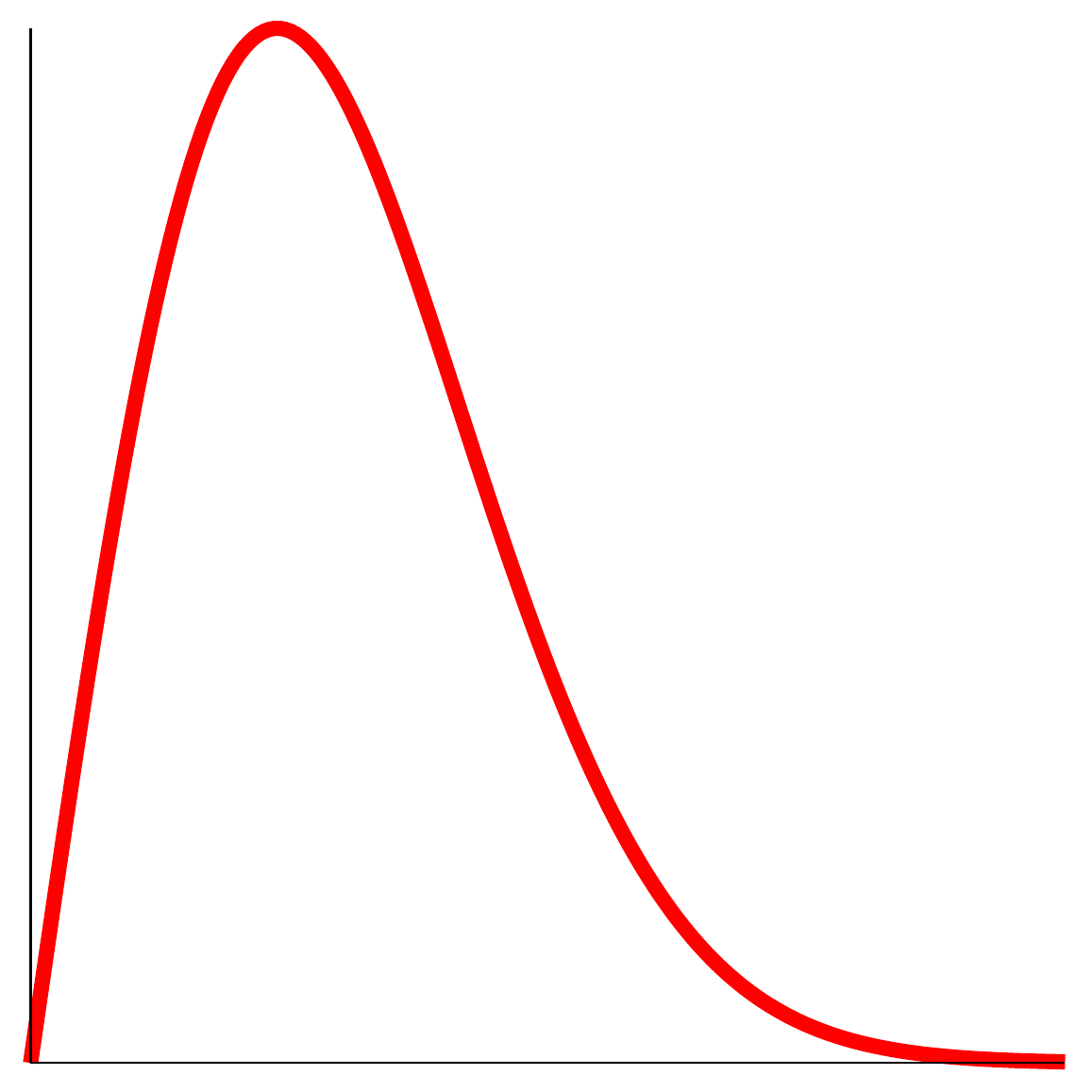} 
\quad
\includegraphics[height=\myheight]{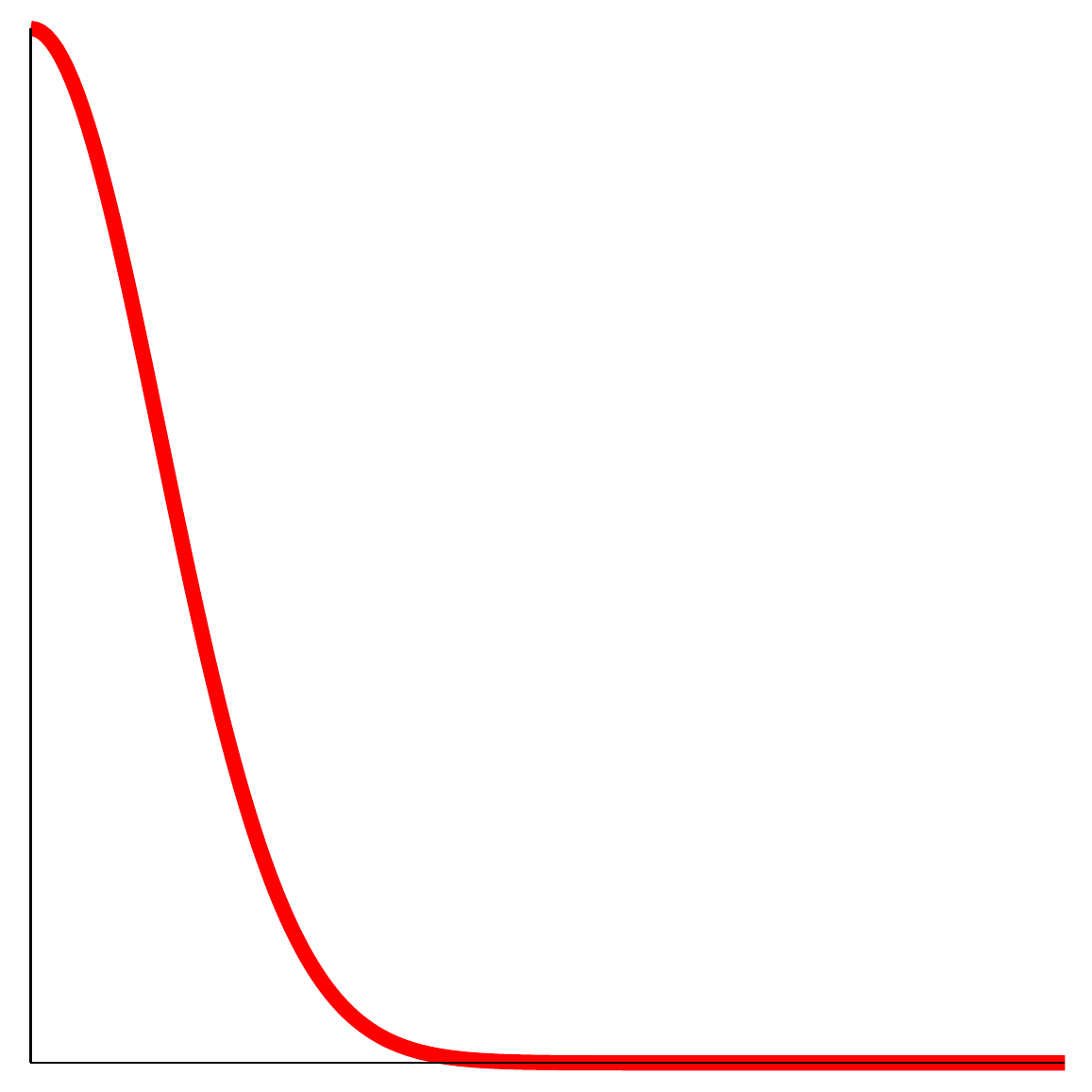}  
\quad
\includegraphics[height=\myheight]{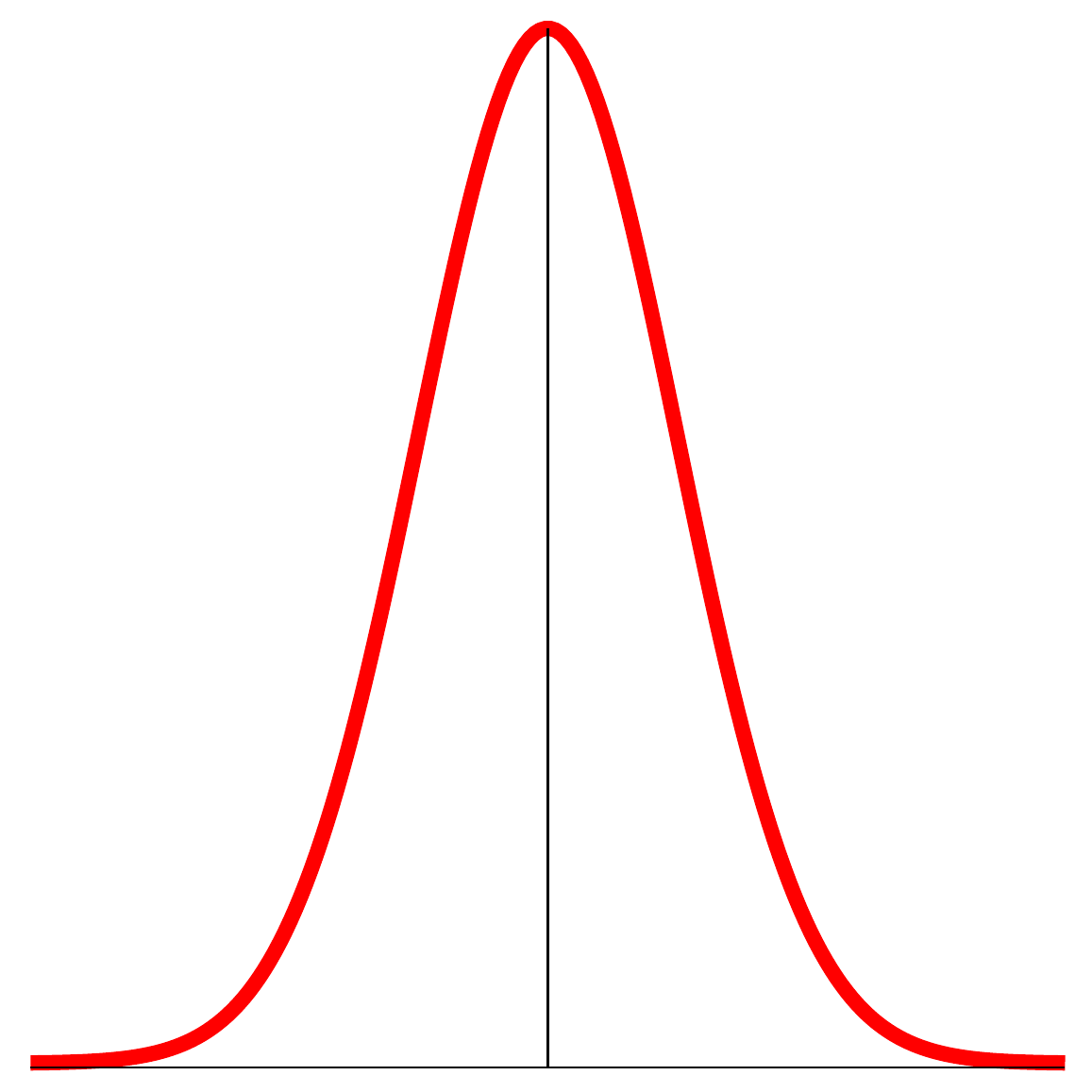}  
\caption{Under the Gibbs model, we get different limit distributions for the random variable $X_n(q)$ associated with statistics counted by a composition scheme: 
from left to right, Boltzmann (e.g., the negative binomial), Mittag-Leffler (e.g., Rayleigh), chi  (e.g., half-normal), Gaussian.
We establish their universality in the next sections.} 
\label{fig:zoo}
\end{figure}

In many (combinatorial) applications of composition schemes, the following assumptions hold~\cite{BKW2021a}: all involved generating functions $G(z)$, $H(z)$, and $F(z)$ have nonnegative coefficients and aperiodic support, are analytic in a $\Delta$-domain with a finite radius of convergence $\rhoF$ and possess singular expansions of the form
\begin{equation}
F(z)=P\left(1-\frac{z}{\rhoF}\right) +  \cF \cdot \left(1-\frac{z}{\rhoF}\right)^{\lambdaF} \left(1+o(1)\right), 
\label{EqSA}
\end{equation}
where $\lambdaF \in \R\backslash \{0,1,2,\dots\}$ is called the \emph{singular exponent} (at $z=\rhoF$),
and where $P(x) \in \mathbb{C}[x]$ is a polynomial (of degree $\geq 1$  for $\lambdaF>1$, of degree $0$ for $0<\lambdaF<1$,
and $P=0$ for $\lambdaF<0$). 
Then, singularity analysis~\cite{FS2009} can often be used to compute the asymptotics of the coefficients.  
Note that the sign of the constant $\cF$ depends on the Puiseux exponent $\lambdaF$; see~\cite[Equation~(8)]{BKW2021a}.

\smallskip

For the reader's convenience, we recall the classical terminology for composition schemes in the following definition.
\begin{definition}[Classification of composition schemes]
\label{def:classification}	
Let $\tauH=H(\rhoH)$. A composition scheme $F(z)=G\big(H(z)\big)$ is called
\emph{subcritical} if it satisfies $\tauH<\rhoG$, \emph{critical} if it satisfies $\tauH=\rhoG$, and \emph{supercritical} if it satisfies $\tauH>\rhoG$.
\end{definition}
We note that each individual class of critical schemes leads to very diverse combinatorial and probabilistic phenomena~\cite{BFSS2001,BD2015,BKW2021a,FS2009,Stufler2022}. 

\smallskip

One typical example of a combinatorial construction  of shape $F(z,u)=G\big(u H(z)\big)$
is given by the sequence construction.
\begin{example}[Sequence of objects]
Given a combinatorial structure $\mathcal{H}$, let $\mathcal{F}=\Seq(\mathcal{H})$.
Using the variable $u$ to encode the core size (i.e., the number of $\mathcal H$ components), one has $F(z,u)=1/(1-u H(z))$.
This is a composition scheme with $G(z)=1/(1-z)$, with $\rhoG=1$ and $\lambdaG=-1$.
\end{example}

We observe that the parameter $q$ can thus directly influence the nature of the underlying singular structure
when the total mass (obtained by $u=1$) changes from $[z^n]F(z,1)$ to $[z^n]F(z,q)$. 
Let us make this more precise for the following class of functions that includes the sequence construction.

\begin{lemma}[Nature and asymptotics of~$q$-enumerated composition schemes] 
\label{Lemma:asympt}
Let a composition scheme $F(z,u)=G\big(u H(z)\big)$ with singular exponents $\lambdaG<0$ and $0<\lambdaH<1$ be given.
Let  $q_c:=\frac{\rhoG}{\tauH }=\frac{\rhoG}{H(\rhoH)}>0$.
The nature of the scheme then splits into three different regimes:
\begin{itemize}
\item for $0<q<q_c$, the scheme  is subcritical;
\item for $q=q_c$, the scheme  is critical; 
\item for $q>q_c$, the scheme  is supercritical. 
\end{itemize}
Accordingly, if one imposes a Gibbs measure on the number of $\mathcal{H}$-components,
this  impacts the asymptotics of their $q$-enumeration $f_n(q)$ as follows: 
\[
f_n(q) \sim 
\begin{cases}
\frac{\cH q G'(q\tauH)}{\Gamma(-\lambdaH)}\rho_H^{-n}n^{-\lambdaH-1}, & \text{ for } \quad 0<q<q_c,\\[1mm]
\cG\Big({-}\frac{\cH}{\tauH}\Big)^{\lambdaG} \frac1{\Gamma(-\lambdaH\lambdaG)}\rho_H^{-n}n^{-\lambdaH\lambdaG-1}, & \text{ for } \quad q=q_c,\\[1mm]
\cG\Big(\frac{q \rho H'(\rho)}{\rhoG}\Big)^{\lambdaG}\frac1{\Gamma(-\lambdaG)}\rho^{-n}n^{-\lambdaG-1}, & \text{ for } \quad q>q_c,
\end{cases}
\]
where, in the last case, $\rho$ is the unique solution of $qH(\rho)=\rhoG$ in the interval $(0,\rhoH)$.
\end{lemma}
\begin{proof}[Proof of Lemma~\ref{Lemma:asympt}]
First, we turn to the nature of the scheme $F_q(z) = G(qH(z))$.
Since $F_q(z)$ has nonnegative coefficients, Pringsheim's Theorem~\cite{FS2009} implies that there is a singularity on the real axis.
Further, $H(z)$ is monotonically increasing on the real axis from $0$ to $\rhoH$, where it attains the value $\tauH = H(\rhoH) < \infty$.
Thus, the nature of $F_q(z)$ depends on the relation between the singularity~$\rhoG$ of $G(z)$ and $q \tauH$ as claimed. 
Next, we look at the singular expansions. 
We start with those of $G(z)$ and $H(z)$.
By the assumptions on $\lambdaG$ and $\lambdaH$ we have
\begin{align}
&&G(z) &\sim \cG\left(1-\frac{z}{\rhoG}\right)^{\lambdaG} 
&& \text{ and } &
H(z) &\sim \tauH + \cH\left(1-\frac{z}{\rhoH}\right)^{\lambdaH}. \label{eqn:expansion-q-H}
&&
\end{align}

In the subcritical regime $0 < q < \frac{\rhoG}{\tauH}$,  
the outer function $G(z)$ is analytic at $q\tauH$ and we combine its expansion with 
the singular expansion of $H(z)$ around $z=\rhoH$ to obtain
\begin{align*}
F_q(z) &= G(q\tauH) + G'(q\tauH)\big(q H(z)-q\tauH\big)(1+o(1))\\
&\sim G(q\tauH)+\cH q G'(q\tauH)\big(1-\frac{z}{\rhoH}\big)^{\lambdaH}.
\end{align*}
Basic singularity analysis~\cite{FS2009} provides the stated expansion.

For the critical regime $q=\frac{\rhoG}{\tauH}$ we obtain
\begin{align*}
F_q(z) 
&\sim \cG\left(1 - q\frac{\tauH}{\rhoG} -  q\frac{\cH}{\rhoG} \left(1-\frac{z}{\rhoH}\right)^{\lambdaH}\right)^{\lambdaG} 
=\cG\left({-}\frac{\cH}{\tauH}\right)^{\lambdaG}\left(1-\frac{z}{\rhoH}\right)^{\lambdaH\lambdaG},
\end{align*}
which leads to the desired result. 

Finally, in the supercritical regime $q > \frac{\rhoG}{\tauH}$, there exists a unique $0<\rho<\rhoH$ such that $q H(\rho)=\rhoG$. 
As $\rho<\rhoG$, we may expand $H(z)$ around $\rho$.
This leads to
\[
q H(z) = q H(\rho)+q H'(\rho)(z-\rho) +\mathcal{O}\left((z-\rho)^2\right).
\]
Note that by the positivity of the coefficients of $H$ one has $H'(\rho)\neq 0$. 
Thus, we have
\[
F_q(z) \sim \cG\left(\frac{q \rho H'(\rho)}{\rhoG}\right)^{\lambdaG} \left(1-\frac{z}{\rho}\right)^{\lambdaG},
\]
and the asymptotic formula for $f_n(q)$ follows by singularity analysis.
\end{proof}

\section{Main theorem: Gibbs models and phase transitions with respect to \texorpdfstring{\boldmath $q$}{q}}\label{sec:Theorem}

The following result is our main theorem. It describes the dependency of the limit law on the parameter $q$.
In this extended abstract, we present only the case of sequence-like schemes.

\begin{theorem}[Gibbs distribution and phase transitions of sequence-like schemes]
\label{the:compoMain}
Let a composition scheme $F(z,u)=G\big(u H(z)\big)$, with singular exponents $\lambdaG < 0$ and $0<\lambdaH<1$, be given. 
Then, the Gibbs distribution of $X_n=X_n(q)$ associated with $F(z,qv)$ has the following limit laws and phase transition diagram that depend on $q_c=\frac{\rhoG}{\tauH }$:
\begin{table}[!htb]
\centering
\begin{tabular}{l@{\hskip 10mm }ccc}
\toprule
Parameter $q$ &$0<q<q_c$ & $q=q_c$ &  $q>q_c$\\
\midrule
Regime & subcritical & critical & supercritical \\ Singularity & $\rhoH$ & $\rhoH$ & $\rhoq<\rhoH$ \\
Singular exponent & $Z^{\lambdaH}$ & $Z^{\lambdaG\lambdaH}$ & $Z^{\lambdaG}$ \\ 
Limit law & discrete & continuous&  continuous \\
& (Boltzmann) & (Mittag-Leffler)  & (Gaussian)\\
\bottomrule
\end{tabular}
\end{table}

\begin{itemize}
\item In the subcritical regime $0<q<q_c$, the random variable $X_n-1$ converges to a discrete distribution, a Boltzmann distribution 
$\mathcal{B}_{G'}(q\tau_h)$     with explicit probability generating function~given~by:
\[
\P(X_n-1=k)\to [v^k] \frac{ G'(v q\tauH)}{G'(q\tauH)}.
\]
In particular, if $G(z)=\frac{1}{(1-z)^{m}}$, the limit law of $X_n-1$ is a negative binomial distribution $\text{NegBin}(m+1,1-q\tauH)$, where $X \sim \text{NegBin}(r,p)$ is defined by $\P(X=k) = \binom{k+r-1}{k} p^r (1-p)^k$ for $k \geq 0$.
\item In the critical regime $q=q_c$, the random variable $X_n/n^{\lambdaH}$ converges in distribution:
\[
\frac{-\cH X_n}{\tauH \, n^{\lambdaH}}\claw X,
\]
where the random variable 
$X$ follows the two-parameter Mittag-Leffler distribution $\operatorname{ML}(\alpha,\beta)$ (with $\alpha:=\lambdaH$ and $\beta:=-\lambdaG\lambdaH$)
that is associated with the density $f_X(x)=\frac{\Gamma(\beta +1)}{\alpha \Gamma(\frac{\beta}{\alpha}+1)}\sum_{n=1}^{\infty} \frac{(-1)^n}{n!\Gamma(-n\alpha)}x^{n +\beta/\alpha-1}$
 determined by its moments $\E(X^r) = \frac{\Gamma(\beta) \Gamma(r+\beta/\alpha) }{\Gamma(\beta/\alpha) \Gamma(\alpha r +\beta)}$.\\
In particular, for $\lambdaG = -1$ and $\lambdaH = \frac12$, $X$ follows the Rayleigh distribution $\mathcal{R}(\sqrt 2)$, where $X \sim \mathcal{R}(\sigma)$ is defined by the density $\frac{x}{\sigma^2} e^{-x^2/(2\sigma^2)}$ for $x \geq 0$.
\item In the supercritical regime $q>q_c$, the centred and normalized random variable $(X_n-\mu_n)/\sigma_n$ converges in distribution to a standard normal distribution $\mathcal{N}(0,1)$, 
where mean~$\mu_n$ and variance~$\sigma_n^2$ are both of order $n$: we have, with $\rho\equiv \rho(q)$ given by $qH(\rho) = \rhoG$,
\[\mu_n \sim \frac{\rhoG}{q \rhoq H'(\rhoq)}\cdot n, \quad \sigma_n^2 \sim \Big( \frac{\rho_G^2}{q^2 \rhoq^2H'(\rhoq)^2} - \frac{\rho_G}{q \rhoq H'(\rhoq)} + \frac{\rho_G^2 H''(\rhoq)}{q^2 \rhoq H'(\rhoq)^3} \Big) \cdot n.\]
\end{itemize}
In particular, the expected value of $X_n$ is for $n \to \infty$ asymptotically equivalent to
\[
\E(X_n)\sim
\begin{cases}
1 + \frac{q \tauH G''(q \tauH)}{G'(q \tauH)}, & \text{ for } \quad 0<q<q_c,\\
\frac{\lambdaG\tauH\Gamma(-\lambdaG\lambdaH)}{\cH \Gamma((1-\lambdaG)\lambdaH)} \cdot n^{\lambdaH}, & \text{ for } \quad q=q_c,\\
\frac{\rhoG}{q \rhoq H'(\rhoq)}\cdot n, & \text{ for } \quad q>q_c.
\end{cases}
\]
\end{theorem}

\begin{proof}[Proof of Theorem~\ref{the:compoMain} (Sketch)]
For $q<q_c$ we are in the subcritical regime and follow the proof of Lemma~\ref{Lemma:asympt}. 
We build on the results of~\cite{FS2009,BKW2021a}. We expand $F(z,qv)$ for $0<v<1$
to obtain
$
F(z,qv) 
\sim G(qv\tauH) + \cH qv G'(qv\tauH)\left(1-\frac{z}{\rhoH}\right)^{\lambdaH}.
$
This implies that the probability generating function satisfies
\[
\lim_{n\to \infty} \E(v ^{X_n(q)})=\lim_{n\to \infty}\frac{[z^n]F(z,qv)}{f_n(q)}
=\frac{v G'(q v\tauH)}{G'(q\tauH)},
\]
leading to a Boltzmann distribution $\mathcal{B}_{G'}(q \tauH)$.
In particular, for $G(z)=1/(1-z)$ we have $G'(z) = 1/(1-z)^2$, leading to a negative binomial distribution.

\smallskip

For $q=q_c$ we are at the critical value and we thus apply~\cite[Theorem 4.1]{BKW2021a} with $\lambdaG < 0$ and $0 < \lambdaH < 1$. 
This yields the stated limit law, as discussed in~\cite[Remark 4.2]{BKW2021a}.

\smallskip

In the supercritical regime for $q>q_c$, our claim results from the approach of Bender~\cite[Propositions~IX.6 and~IX.7]{FS2009}.
The singularity $\rho = \rho(qv)$ becomes an analytic function of $v$ while the nature of the singularity remains unchanged for $v$ in a sufficiently small neighbourhood of $1$. 
In particular, the expected value, the variance, and the normal limit law follow by an application of Hwang's quasi-power theorem~\cite{Hwang1998,FS2009}.
\end{proof}
\begin{figure}[htb]
\newcommand{\myheight}{4.4cm}
\centering
\includegraphics[height=\myheight]{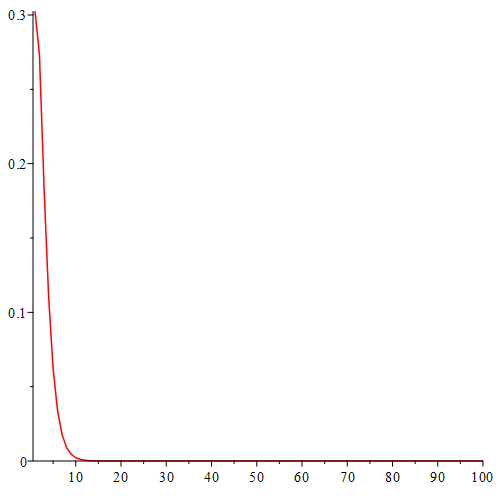}\quad\includegraphics[height=\myheight]{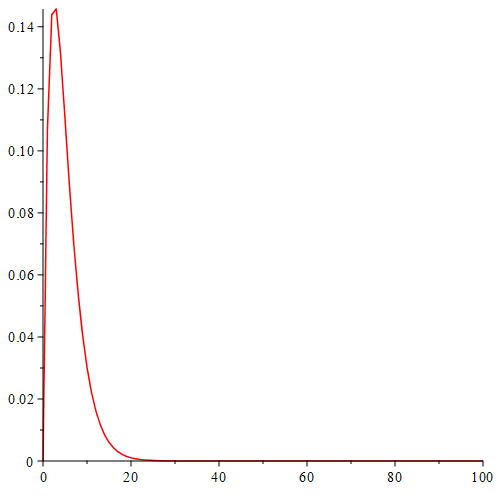}\quad\includegraphics[height=\myheight]{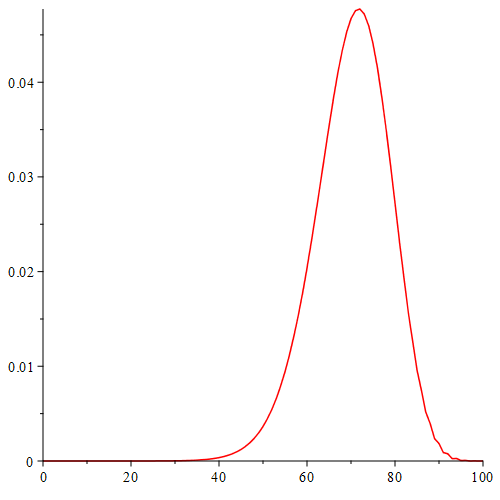}\\
$q=1$   \hspace{4cm}  $q=1.5$  \hspace{4cm}    $q=3$
\caption{The distribution (with  the histogram interpolated to a curve) 
of returns to 0 in Motzkin excursions (as analysed in Section~3), for walks of length $n=100$, under the Gibbs measure of weight $q$.
While the distribution for a finite $n$ evolves continuously when one increases $q$,
it exhibits a phase transition when one  increases~$n$ (at $q_c=3/2$ for this example).
As stated in Theorem~\ref{the:compoMain}, 
it converges  to a negative binomial distribution for $q<q_c$, to a Rayleigh distribution for $q=q_c$, and to a Gaussian distribution for $q>q_c$.}
\end{figure}
It is worth pointing out that the constant in the asymptotics of $\sigma_n^2$ in the supercritical case is always strictly positive. Thus degenerate limit laws are not possible. This is shown in the following. The proof follows the lines of~\cite[Proposition 23]{Kropf2017}.

\begin{proposition}[Positivity of the variance constant]
In the supercritical regime $q > q_c$ of Theorem~\ref{the:compoMain}, we have
\[\frac{\rho_G^2}{q^2 \rhoq^2H'(\rhoq)^2} - \frac{\rho_G}{q \rhoq H'(\rhoq)} + \frac{\rho_G^2 H''(\rhoq)}{q^2 \rhoq H'(\rhoq)^3} > 0,\]
thus $\frac{\sigma_n^2}{n}$ converges to a constant that is strictly positive.
\end{proposition}

\begin{proof}
For $v$ in a suitably small neighbourhood of $1$, the singularity $\rho = \rho(qv)$ of $F(z,qv)$ becomes an analytic function of $v$ as mentioned in the proof of Theorem~\ref{the:compoMain}. It is implicitly determined by $qv H(\rho(qv)) = \rho_G$. Singularity analysis gives us the asymptotic formula
\[f_n(qv) = [z^n] F(z,qv) \sim \cG\Big(\frac{qv \rho(qv) H'(\rho(qv))}{\rhoG}\Big)^{\lambdaG}\frac1{\Gamma(-\lambdaG)}\rho(qv)^{-n}n^{-\lambdaG-1};\]
see Lemma~\ref{Lemma:asympt}. For the moment generating function of the variable $X_n(q)$, this means that
\begin{equation}\label{eq:mgf-asymp}
\E(e^{t X_n(q)}) = \frac{f_n(qe^t)}{f_n(q)} \sim \exp \big( a(t) n + b(t) \big),
\end{equation}
uniformly for $t$ in a suitable neighbourhood of $0$, where 
\[a(t) = \log \rho(q) - \log \rho(q e^t),\qquad b(t) = \lambdaG \Big( t + \log \frac{\rho(qe^t) H'(\rho(qe^t))}{\rho(q) H'(\rho(q))} \Big).\]
Note that $a(0) = b(0) = 0$ and $a'(0) = -\frac{q \rho'(q)}{\rho(q)} = \frac{\rhoG}{q \rho(q) H'(\rho(q))}$ (the latter identity follows by implicit differentiation from $q H(\rho(q)) = \rhoG$). Now let $k$ be the smallest positive integer greater than or equal to $2$ such that $a^{(k)}(0) \neq 0$, provided that such an integer exists. Substitute $t = s n^{-1/k}$ in~\eqref{eq:mgf-asymp} and apply the Taylor expansion of $a(t)$ and $b(t)$ around $0$ to obtain
\[\E \Big( \exp \Big( \frac{s (X_n(q)-a'(0)n)}{n^{1/k}} \Big)\Big) = \exp \Big( \frac{a^{(k)}(0)}{k!} s^k + o(1) \Big).\]
By L\'evy’s continuity theorem, this would mean that $ \frac{X_n(q)-a'(0)n}{n^{1/k}}$ converges in distribution to a random variable $X$ with moment generating function 
\[M(s) = \E (e^{sX}) = \exp \Big( \frac{a^{(k)}(0)}{k!} s^k \Big).\]
However, such a random variable can only exist for $k=2$: otherwise, it would have second moment~$0$, thus be almost surely equal to $0$ and have moment generating function $M(s) = 1$. If $k = 2$, then
\begin{align*}
0 \neq a''(0) &= -\frac{q^2 \rho''(q)}{\rho(q)}+\frac{q^2 \rho'(q)^2}{\rho(q)^2}-\frac{q \rho'(q)}{\rho(q)} \\
&= \frac{\rho_G^2}{q^2 \rho(q)^2H'(\rho(q))^2} - \frac{\rho_G}{q \rho(q) H'(\rho(q))} + \frac{\rho_G^2 H''(\rho(q))}{q^2 \rho(q) H'(\rho(q))^3},
\end{align*}
and we are done (here, implicit differentiation of $q H(\rho(q)) = \rhoG$ is used again in the second step). It only remains to exclude the possibility that there is no $k \geq 2$ such that $a^{(k)}(0) \neq 0$. Then $a(t)$ must be a linear function: $a(t) = \kappa t$ for some constant $\kappa$, thus $\rho(q e^t) = \rho(q) e^{-\kappa t}$. This would have to hold for $t$ in a neighbourhood of $0$. But in view of the implicit equation $q e^t H(\rho(q e^t)) = \rhoG$, this would imply that the function $H$ is given by $H(z) = \frac{\rhoG}{q} (z/\rho(q))^{1/k}$, contradicting our assumptions.
\end{proof}

\section{Applications: phase transitions from negative binomial to Rayleigh to Gaussian}\label{sec:Applications}
We start our list of applications with the case of fixed-point-biased permutations avoiding a pattern of length three, 
whose asymptotic behaviour was recently considered in~\cite{ChelikavadaPanzo2023}. 
Then we consider several instructive examples from the theory of lattice paths, namely
the returns to zero in Dyck bridges and Dyck excursions, as well as in Motzkin bridges and Motzkin excursions. 
Thereby, we add to the existing classical results (that is, for the uniform distribution) 
the phase transitions stemming from $q$-enumeration (that is, for the Gibbs distribution on the parameter). 
We also consider the number of boundary interactions in some quarter-plane walk models,
and the number of contacts between two paths in some watermelon models.

\subsection{Fixed-point-biased permutations avoiding a pattern of length three}
Let $\mathcal{S}_n(p)$ be the set of permutations of $1,2,\ldots,n$ that avoid  a given pattern $p$ 
(where, as usual, the elements of the pattern $p$ need not be contiguous in the permutation; see Figure~\ref{fig:permutation}). 
The generating function counting the statistic ${\operatorname{fp}(\sigma)}$ (number of fixed points) of such permutations for the pattern $321$ was obtained by 
Vella~\cite[Theorem~2.13]{Vella2003} 
and for the other two patterns $132$ and $213$ by 
Elizalde~\cite[Theorem~3.5]{Elizalde04}.
It is in all three cases equal to
\[
F(z,u)=1+\sum_{n=1}^{\infty}\, \sum_{\sigma\in\mathcal{S}_n(p)}u^{\operatorname{fp}(\sigma)}z^n=\frac{2}{1+2(1-u)z+\sqrt{1-4z}}.
\]
\begin{figure}[hbt!]
\centering
\begin{tikzpicture}[scale=0.26]	
\newcommand{\msize}{13}
	\draw[step=1cm,lightgray,very thin] (1,1) grid (\msize,\msize);
		
\newcommand{\addpoint}[1]{\filldraw[black] #1+(0.5,0.5) circle (9pt);
	}
	
\newcommand{\addpointred}[1]{\filldraw[red] #1+(0.5,0.5) circle (8pt);
	}
	\addpoint{(1,2)}
	\addpoint{(2,1)}
	\addpointred{(3,3)}
	\addpoint{(4,5)}
	\addpoint{(5,6)}
	\addpoint{(6,4)}
	\addpointred{(7,7)}
	\addpointred{(8,8)}
	\addpoint{(9,10)}
	\addpoint{(10,12)}
	\addpoint{(11,9)}  
	\addpoint{(12,11)}
	\end{tikzpicture} 
\caption{An example of a 321-avoiding permutation of length $12$ with $3$ fixed points marked by red dots.
(A permutation $\pi=(\pi_1,\dots,\pi_n)$ avoids the pattern $p=321$ if there is no triplet $1\leq i<j<k \leq n$ such that $\pi_k<\pi_j<\pi_i$.)} 
\label{fig:permutation}
\end{figure}
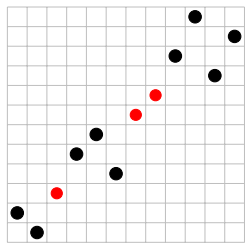 

Later, Chelikavada and Panzo~\cite{ChelikavadaPanzo2023} used this generating function to establish a phase transition which we rederive now.
\begin{theorem}[Phase transition for fixed-point-biased permutations] 
The limit Gibbs distribution of~the fixed-point statistic in  permutations avoiding any given pattern $p \in \{132, 321, 213\}$ has a phase transition
with critical value $q_c=3$:
\begin{center}
\begin{tabular}{l@{\hskip 10mm }ccc}
\toprule
Parameter $q$ &$0<q<3$ & $q=3$ &  $q>3$\\
\midrule
Limit law of $X_n(q)$ & Negative binomial  & Rayleigh & Gaussian\\
& $\nb(2,1-q/3)$ & $\mathcal{R}(\sqrt{2})$ & $\mathcal{N}(0,1)$ \\  
\bottomrule
\end{tabular}
\end{center}
\end{theorem}

\begin{proof}
In order to deduce the phase transitions from Theorem~\ref{the:compoMain},
we write $F(z,u)$ as a sequence of components $H(z)$ marked with $u$ 
\[
F(z,u)=\frac{H(z)}{z}\cdot \frac{1}{1-uH(z)}=
\frac{1}{uz}\cdot \frac{1}{1-uH(z)}-\frac{1}{uz},
\text{\quad where \quad }
H(z)=\frac{2z}{1+2z+\sqrt{1-4z}}.
\]
Here, $F(z,u)$ is not in the shape of the composition scheme~\eqref{Scheme:1}, but very close. 
For the limit law, these perturbing factors are irrelevant, since
by Lemma~\ref{lem:basicProb} one has for $n,k \geq 0$
\begin{equation*}
\P(\Xnq=k) 
= \frac{[z^n u^k]F(z,qu)}{[z^n]F(z,qu)} 
= \frac{[z^{n+1} u^{k+1}]G(uqH(z))}{[z^{n+1}]G(uqH(z))},
\text{\quad where \quad }
G(z)=\frac{1}{1-z}.
\end{equation*}
Thus, one has $\rhoG=1$, $\rhoH=\frac14$ and $\tauH=H(\rhoH)=\frac13$. 
Theorem~\ref{the:compoMain} then implies the phase transition at $q_c=\rhoG/\tauH=3$
with a Boltzmann distribution (simplifying here to a NegBin),
a Mittag-Leffler distribution (simplifying here to a Rayleigh distribution since $\lambdaH=\frac12$), and a Gaussian distribution. 
\end{proof}

\subsection{Returns to zero in Dyck and Motzkin paths}
We consider two classical directed models: Dyck and Motzkin paths.
Dyck paths consist of the steps up $(1,1)$ and down $(1,-1)$,
while Motzkin paths allow additionally some horizontal steps $(1,0)$.
They are called \emph{bridges} if they start at $(0,0)$ and end at $(2n,0)$,
and \emph{excursions} if it is additionally required that they never cross the $x$-axis.

The random variable $X_{n}$, counting the number of returns to the $x$-axis in a random excursion and bridge of size~$2n$, 
is a well-studied object~\cite{FS2009,bafl02}, leading for excursions to a negative binomial limit law for $X_n$
and leading for bridges to a Rayleigh limit law for $X_n/\sqrt{n}$. 
Note that the root degree in plane trees behaves the same due to a classical bijection between plane trees and Dyck excursions.
What happens when we give a Gibbs weight to the number of returns? The following theorem answers  this question.

\begin{theorem}[$q$-enumerations: limit laws for returns to zero]
\label{the:DyckMotzkin}
Under the Gibbs model,  the number $X_{n} = X_n(q)$ of returns to zero in Dyck excursions and bridges of length $2n$, 
in Motzkin excursions and in bridges of length $n$, has, for $n\rightarrow +\infty$, 
 a phase transition at $q_c$ and  follows, after rescaling, either a negative binomial, Rayleigh, or Gaussian distribution. 

\begin{minipage}{0.527\textwidth}
\smallskip
\centering
\begin{tabular}{cc}
\toprule
Parameter $q$ & Limit law \\
\midrule
$0<q<q_c$ & $X_n-1 \claw  \nb(2,1-q\tauH)$ \\
$q=q_c$ &  $\frac{-\cH}{\tauH}\frac{X_n}{\sqrt{n}} \claw \ray(\sqrt 2)$\\
$q>q_c$ & $\frac{X_n-\mu \cdot n}{\sigma \cdot \sqrt{n}} \claw  \Nc(0,1)$\\
\bottomrule
\end{tabular}
\end{minipage}
\begin{minipage}{0.43\textwidth}
\begin{align*}
q_c &= 
\begin{cases}
2 & \text{for Dyck excursions},\\
1 & \text{for Dyck bridges},\\
\frac{3}{2} & \text{for Motzkin excursions},\\
1 & \text{for Motzkin bridges}.
\end{cases}
\end{align*}
\end{minipage}
\medskip
\\
Here, $\tauH$ is $\frac12,1,\frac23,1$, and $\cH$ is $-\frac12,-1,-\frac1{\sqrt3},-\frac2{\sqrt{3}}$ for Dyck excursions, Dyck bridges, Motzkin excursions and Motzkin bridges, respectively.
\end{theorem}
\pagebreak 

\begin{proof}
Cutting each time the path returns to the $x$-axis~\cite[p.~636]{FS2009}, one directly sees that the generating functions $D(z)$ and $B_D(z)$ of Dyck excursions and bridges, respectively, satisfy the relations
\[D(z) = \frac{1}{1-z^2D(z)} = \frac{1 - \sqrt{1-4z^2}}{2 z^2}
\qquad \text{ and } \qquad
B_D(z) = \frac{1}{1-2z^2D(z)} = \frac{1}{\sqrt{1-4z^2}}.\]
The generating functions $D(z,u)$ and $B_D(z,u)$ of Dyck excursions and bridges marking the number of returns are
\begin{align}\label{Du}
D(z,u) = \frac{1}{1-z^2uD(z)}
\qquad \text{ and } \qquad
B_D(z,u) = \frac{1}{1-2z^2uD(z)}.
\end{align}
In both cases we recognize a composition scheme~\eqref{Scheme:1} with $G(z)=\frac{1}{1-z}$ and $H(z) = z^2D(z)$ for excursions and $H(z)=2z^2 D(z)$ for bridges.
Therefore, we can readily apply Theorem~\ref{the:compoMain}. 

Motzkin excursions have the generating function
\[
M(z)
=\frac{1-z - \sqrt{ (1+z)(1-3z)}}{2z^2},
\]
with dominant singularity $\rho=\frac13$~\cite{FS2009}. 
With the same ideas as for Dyck paths, we directly get the bivariate generating functions $M(z,u)$ and $B_M(z,u)$ of excursions and bridges, respectively, as
\begin{equation}\label{Mu}
M(z,u)=\frac{1}{1- z u\big(1+z M(z)\big)}
\qquad \text{ and } \qquad
B_M(z,u) = \frac{1}{1- z u\big(1+2z M(z)\big)}.
\end{equation}
Again, we recognize the composition scheme~\eqref{Scheme:1} with $G(z)=\frac{1}{1-z}$ and $H(z) = z(1+zM(z))$ for excursions and $H(z)=z(1+2zM(z))$ for bridges.
This leads again to similar phase transitions.
\end{proof}

\begin{remark}[Weighted paths]
\newcommand{\ww}{p}
Let $\ww_{\sminus 1}, \ww_0, \ww_{1} \geq 0$ be the weights of the steps $(1,{\sminus}1)$, $(1,0)$, $(1,1)$, respectively. 
The weight of a path is the product of its weights.
Weighted Dyck excursions and bridges behave exactly as unweighted ones, as each path of length $2n$ has a weight $(\ww_{\sminus 1} \ww_{1})^n$.
However, weighted Motzkin excursions and bridges behave differently. 
With the same techniques it is easy to show that for weighted Motzkin excursions, the phase transition occurs at 
\begin{align*}
q_c 
= \frac{\ww_0 + 2 \sqrt{\ww_{\sminus 1} \ww_{1}}}{\ww_0 + \sqrt{\ww_{\sminus 1} \ww_{1}}}
= 1 + \frac{1}{1 + \frac{\ww_0}{\sqrt{\ww_{\sminus 1} \ww_{1}}}}
\in (1,2].\end{align*}	
Finally, for weighted Motzkin bridges the phase transition again always occurs at $q_c=1$, because $\tauH=1$ is independent of the weights. 
\end{remark}

\subsection{Boundary contacts for quarter-plane walks}
A direct byproduct of our results are phase transitions for Hadamard models of quarter-plane walks.
Beaton, Owczarek, and Rechnitzer~\cite{BOR2019} initiated the study of quarter-plane walks with wall interactions. 
We are interested in walks restricted to the quarter plane, starting and ending at the origin, and their interaction with the walls
(that is, their number of contacts with $x$- or $y$-axis). 
\pagebreak

It turns out that in many models the generating functions are rather complicated~\cite{BOR2019}, 
but for three families of walks called \emph{Hadamard models}, the analysis of contacts can be done in a fairly simple way. 
Such models are enumerated by a \emph{Hadamard product} of generating functions
\begin{equation*}
A(z)\odot B(z):=\sum_{n\ge 0}a_n b_n z^n,
\qquad 
\text{where } A(z)=\sum_{n\ge 0}a_n z^n
\text{\quad and  \quad} B(z)=\sum_{n\ge 0}b_nz^n.
\end{equation*}

They correspond to the diagonal, diabolo, and king walk models with stepsets 
\begin{center}\Diagonal, \Diabolo, \King,\end{center}
respectively (for the last two models, one has in fact a slight variant of the usual diabolo or king models: here we allow additionally the step $(0,0)$, as indicated by the centre dot in the stepset representation).
Walks of length $2n$ in the quarter plane that start and end at $(0,0)$ decompose into two independent directed excursions of length $n$.
Therefore, these models are in bijection with pairs of Dyck and Motzkin excursions.
This is summarized in the following table,
where $D(z,u)$ and $M(z,u)$ are the generating functions from~\eqref{Du} and \eqref{Mu}, respectively.

\begin{center}
\begin{tabular}{ccccc}
\toprule
Model & Steps & Generating function $Q(z,u_1,u_2)$ & Sequence $Q_{2n}$ & OEIS\\  
\midrule
Diagonal & \Diagonal & $D(z,u_1)\odot D(z,u_2)$ & $C_n\cdot C_n$ & \OEISs{A001246} \\
Diabolo & \Diabolo & $D(z,u_1)\odot M(z,u_2)$ & $C_n\cdot M_n$ &\OEISs{A151362} \\
King & \King & $M(z,u_1)\odot M(z,u_2)$ & $M_n\cdot M_n$ & \OEISs{A133053}\\
\bottomrule
\end{tabular}
\end{center}
\smallskip

{}

The limit laws for the number of wall interactions with the $x$-axis or $y$-axis depend on the particular values of the $q$-enumerations; 
compare with the results in Theorem~\ref{the:DyckMotzkin}.
More precisely, we get the following proposition. By symmetry it also translates to $y$-axis contacts in the missing cases.

\begin{theorem}[Boundary interactions for some quarter-plane walks]
The number of $x$-axis contacts of diagonal walks and the number of $y$-axis contacts of diabolo walks follows the law of the q-enumeration of Dyck returns to $0$.
The number of $x$-axis contacts of diagonal, diabolo, and king walks follows the law of the q-enumeration of Motzkin returns to $0$.
Accordingly, the phase transitions are the same as in Theorem~\ref{the:DyckMotzkin}.
\end{theorem}

\begin{proof}[Proof (sketch)]
In the diagonal model, the generating function of the $x$-axis contacts is equal to $D(z,qu) \odot D(z)$.
Therefore, by Lemma~\ref{lem:basicProb} we have
\begin{equation*}
\P(\Xnq=k)= \frac{[z^n u^k]D(z,qu) \odot D(z)}{[z^n]D(z,q) \odot D(z)} = \frac{[z^n u^k]D(z,qu)}{[z^n]D(z,q)}.
\end{equation*}
Thus, the result follows directly from Theorem~\ref{the:DyckMotzkin}.
The other models follow in the same fashion.
\end{proof}

\begin{remark}[$x$-axis plus $y$-axis contacts]
The law of the number of $x$-axis plus $y$-axis contacts is more involved,
as it requires the study of the analytic behaviour of Hadamard products of the shape  $(1-z)^a \odot (1-z)^b$.
These products were studied by Fill, Flajolet, and Kapur, who gave the corresponding Puiseux expansions; see~\cite[Proposition 8]{FillFlajoletKapur2005}.
Note that, as pinpointed in~\cite{FillFlajoletKapur2005}, the case $a+b$  integer requires the use of additional hypergeometric identities. 
\end{remark}

\subsection{Friendly two-watermelons without wall: contacts and returns}
We consider a pair of directed walkers with Dyck steps $(1,-1)$ and $(1,1)$.
The walkers start and end at the same point, may meet and share edges but not cross.
Such walker configurations $\mathcal{W}$ are also called friendly two-watermelons~(see the related work of Roitner~\cite{Roitner2020b}
and Krattenthaler, Guttmann, and Viennot~\cite{KrattenthalerGuttmannViennot2000}).
We are interested in walks of length $n$ and the number of contacts $C$ of the two walkers.
A contact in a two-watermelon is a point (not counting the starting point) where both paths occupy the same lattice point. (See Figure~\ref{fig:twowatermeloncontacts}.)

\begin{theorem}[Phase transition for contacts in friendly two-watermelons]\label{theo:2watermelons}
With the renormalizations of Theorem~\ref{the:compoMain},
the limit Gibbs distribution of the number of contacts in friendly two-watermelons has a phase transition
with critical value $q_c=4/3$:  
\begin{center}
\begin{tabular}{l@{\hskip 10mm }ccc}
\toprule
Parameter $q$ &$0<q<\frac{4}3$ & $q=\frac{4}3$ &  $q>\frac{4}3$\\
\midrule
Limit law of $X_n(q)$ & Negative binomial & Rayleigh  &  Gaussian\\
& $\nb(2,1-\frac{3}{4}q)$ & ${\mathcal R}(\sqrt{2})$ &  $\mathcal{N}(0,1)$\\
\bottomrule
\end{tabular}\end{center}
\end{theorem}

\begin{proof}
Let $F(z,u)$ denote the bivariate generating function of friendly two-watermelons with respect to the number of contacts:
\[
F(z,u)=\sum_{w\in \mathcal{W}}z^{|w|}u^{C(w)}.
\]
This generating function was determined by Roitner~\cite{Roitner2020b}, using a reduction to weighted Motzkin paths. It is given by
\begin{equation}
\label{eqn:VR1} 
F(z,u)=\frac{1}{1-u\big(z^2 W(z)+2z\big)},\qquad W(z)=\frac{1-2z-\sqrt{1-4z}}{2z^2}.
\end{equation}
Under the uniform distribution model, Roitner also obtained a discrete limit law for the parameter (number of contacts). 
We note in passing that closely related problems in families of osculating walkers have been considered before by Bousquet-Mélou~\cite{BousquetMelou06}.
Under the Gibbs distribution model, we recognize a composition scheme 
with $G(z)=1/(1-z)$, $\rhoG=1$, and $H(z)= z(zW(z)+2)$ so that $\tauH=\frac34$.
Therefore, Theorem~\ref{the:compoMain} applies: we get the 3 phases, with the critical value $q_c=\frac{\rhoG}{\tauH }=\frac43$.
\end{proof}

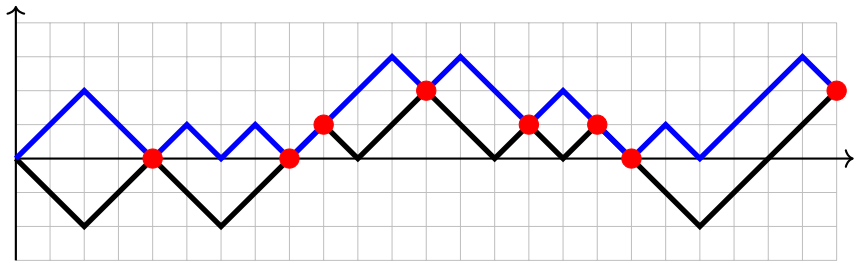
\begin{figure}[htb!]
\centering
\begin{tikzpicture}[scale=0.45]
    
\newcommand{\pathLength}{24}
		
\newcommand{\maxup}{4}
		
\newcommand{\maxdown}{-3}
        \draw[step=1cm,lightgray,very thin] (0,\maxdown) grid (\pathLength,\maxup);
        \draw[thick,->] (0,0) -- (\pathLength+0.5,0) node[anchor=north west] {};
        \draw[thick,->] (0,\maxdown) -- (0,\maxup+0.5) node[anchor=south east] {};

\newcommand{\up}{ -- ++(1,1) }
		
\newcommand{\down}{ -- ++(1,-1) }
		
    \draw[line width=2pt,black] (0,0)
    \down
    \down
    \up
    \up
    \down
    \down
    \up
    \up
    \up
    \down
    \up
    \up
    \down
    \down
    \up
    \down
    \up
    \down
    \down
    \down
    \up
    \up
    \up
    \up
	;
	
    \draw[line width=2pt,blue] (0,0)
    \up
    \up
    \down
    \down
    \up
    \down
    \up
    \down
    \up
    \up
    \up
    \down
    \up
    \down
    \down
    \up
    \down
    \down
    \up
    \down
    \up
    \up
    \up
    \down
	;
    
    \filldraw[red] (4,0) circle (8pt);
    \filldraw[red] (8,0) circle (8pt);
    \filldraw[red] (9,1) circle (8pt);
    \filldraw[red] (12,2) circle (8pt);
    \filldraw[red] (15,1) circle (8pt);
    \filldraw[red] (17,1) circle (8pt);
    \filldraw[red] (18,0) circle (8pt);
    \filldraw[red] (\pathLength,2) circle (8pt);
    
\end{tikzpicture}
\caption{An example of a friendly two-watermelon without wall of length $24$ with $8$ contacts marked by red dots.
Under the Gibbs model where such an object is given the weight $q^8$, the distribution of the number of contacts 
then depends on the value of~$q$, according to the phase transition given in Theorem~\ref{theo:2watermelons}.
} 
\label{fig:twowatermeloncontacts}
\end{figure} 

\section{Extensions to other constructions: new phase transitions from negative binomial to chi to Gaussian}\label{sec:Extensions}
It is appealing to extend Theorem~\ref{the:compoMain} to even more general composition schemes such as
\begin{equation}\label{extendedscheme}
F(z,u)=M(z)\cdot G\big(u H(z)\big),
\end{equation}
where an additional factor $M(z)$ appears. This extended scheme 
is of interest as it captures many classical combinatorial structures: 
some families of trees or lattice paths (meanders), Pólya urns, 
and other probabilistic processes like the Chinese restaurant model.
In  its critical phase, this scheme was recently analysed in~\cite{BKW2021a} 
under the uniform distribution model for the associated combinatorial structures.
This extends the work of pioneers like Bender, Flajolet, Soria, Drmota, Hwang, Gourdon~\cite{Bender1973,FlajoletSoria93,DrmotaSoria95,DrmotaSoria97,Hwang1998,Gourdon98},
later synthesized in~\cite{FS2009}.

In the long version of this article, we analyse this extended scheme under the Gibbs measure model,
and we show that the phase transition for Gibbs distributions leads, in some cases,
to the 3-parameter Mittag-Leffler distribution introduced in~\cite{BKW2021a}. 
Further examples of combinatorial problems involving an extended scheme 
are the root degree in two-connected outerplanar graphs 
(see Drmota, Giménez, and Noy~\cite[Theorem 3.2]{DrmotaGimenezNoy2011}),
the returns to zero in coloured walks, and the number of wall interactions in watermelons.  
We now analyse the phase transitions for these last two models.

\subsection{Number of wall contacts in watermelons}
An $m$-watermelon of length $2n$ consists of $m$ walkers
moving from $(0,2i-2)$ to $(2n,2i-2)$, $1\le i \le m$, 
where every walker may either take an up step $(1,1)$ or a down step $(1,-1)$, 
but walkers are not allowed to occupy the same position (they are thus called \emph{vicious}).
One says that the watermelon has \emph{a wall} 
if the $x$-axis serves as a barrier which the lowest walker may touch but not cross.  (See Figure~\ref{fig:viciouswalkers}.)
Watermelons were introduced by Fisher~\cite{Fisher1984} for modelling wetting and melting.
We refer the reader to the work of Krattenthaler, Guttmann, and Viennot~\cite{KrattenthalerGuttmannViennot2000,Krattenthaler2006}
or Feierl~\cite{Feierl2009,Feierl2012,Feierl2014} for more results on watermelons and related problems.
\begin{figure}[hbt!]
\centering
\begin{tikzpicture}[scale=0.31]
    
\newcommand{\pathLength}{24}
		
\newcommand{\maxup}{8}
		
\newcommand{\maxdown}{0}
        \draw[step=1cm,lightgray,very thin] (0,\maxdown) grid (\pathLength,\maxup);
        \draw[thick,->] (0,0) -- (\pathLength+0.5,0) node[anchor=north west] {};
        \draw[thick,->] (0,\maxdown) -- (0,\maxup+0.5) node[anchor=south east] {};

\newcommand{\up}{ -- ++(1,1) }
		
\newcommand{\down}{ -- ++(1,-1) }
		
    \draw[line width=2pt,black] (0,0)
    \up
    \up
    \down
    \down
    \up
    \down
    \up
    \down
    \up
    \up
    \up
    \down
    \up
    \down
    \down
    \up
    \down
    \down
    \up
    \down
    \up
    \up
    \down
    \down
	;

    \draw[line width=2pt,blue] (0,2)
    \up
    \up
    \down
    \up
    \down
    \down
    \up
    \up
    \up
    \down
    \up
    \up
    \down
    \down
    \up
    \down
    \up
    \down
    \down
    \down
    \up
    \up
    \down
    \down
	;
	
    \draw[line width=2pt,green] (0,4)
    \up
    \up
    \up
    \down
    \down
    \up
    \down
    \up
    \up
    \down
    \up
    \up
    \down
    \down
    \up
    \down
    \up
    \down
    \down
    \up
    \up
    \down
    \down
    \down
	;
    
    \filldraw[red] (0,0) circle (8pt);
    \filldraw[red] (4,0) circle (8pt);
    \filldraw[red] (6,0) circle (8pt);
    \filldraw[red] (8,0) circle (8pt);
    \filldraw[red] (18,0) circle (8pt);
    \filldraw[red] (20,0) circle (8pt);
    \filldraw[red] (24,0) circle (8pt);
\end{tikzpicture}
\caption{A 3-watermelon (with a wall) of length $24$ with $7$ $x$-axis 
contacts marked by red dots. 
Under the Gibbs model where such an object is given the weight $q^7$, 
the distribution of the number of contacts 
then depends on the value of~$q$, according to the phase transition given in Theorem~\ref{theo:mwatermelons}.
} 
\label{fig:viciouswalkers}
\end{figure}
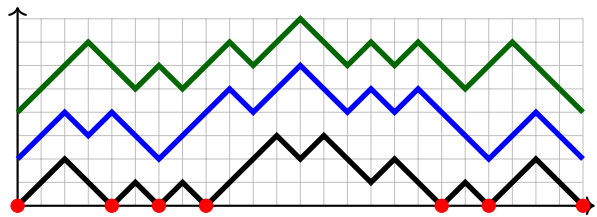 

We are interested in the number of $x$-axis contacts of the lowest walker. 
For $m=1$ this reduces to Theorem~\ref{the:DyckMotzkin} on Dyck excursions.  For $m\geq 1$, we get the following theorem.

\begin{theorem}[Phase transition for wall contacts]\label{theo:mwatermelons}
The limit Gibbs distribution of the number of $x$-axis contacts in $m$-watermelons with a wall 
has a phase transition at 
$q_c=2$:  
\begin{center}
\begin{tabular}{l@{\hskip 10mm }ccc}
\toprule
Parameter $q$ &$0<q<2$ & $q=2$ &  $q>2$\\
\midrule
Limit law of $X_n(q)$ & Negative binomial & Chi distribution  & Gaussian\\
& $\nb(2m,1-\frac{q}{2})$ & $  \frac{X_n}{\sqrt{n}} \claw \chi(2m)$ & $\mathcal{N}(0,1)$ \\
\bottomrule
\end{tabular}\end{center}
\end{theorem}

\begin{proof}[Proof (sketch).]
The $q$-enumeration of the number of contacts in $m$-watermelons of length $2n$ ($n>0$)  with a wall was
given by Krattenthaler~\cite[Theorem 4]{Krattenthaler2006}:
\begin{equation}\label{Kratproduct}
f_{n}(q)=\frac{(n-1)!\prod_{i=0}^{m-1}(2i+1)!\prod_{i=0}^{m-2}(2n+2i)!}{\prod_{i=0}^{2m-2}(n+i)!}
\sum_{\ell=2}^{n+1}   \binom{2n-\ell}{n-1}\binom{\ell+2m-3}{\ell-2} q^{\ell}.
\end{equation}
In the following, we denote by $X_n$ the random variable counting the number of wall contacts, where we drop the dependence on $m$. Its probability generating function $\E(q^{X_n})$ satisfies
\begin{equation}
\label{eqn:Kratt1}
\E(q^{X_{n}})=\frac{f_n(q)}{f_n(1)}.
\end{equation}
Let $F(z,q)$ denote the generating function of the numerator of the reduced fraction~\eqref{eqn:Kratt1}, i.e.,
\vspace{-1em}
\begin{equation}
\label{eqn:Krattenthaler2}
F(z,q)=\sum_{n> 0}z^n \sum_{\ell=2}^{n+1}\binom{2n-\ell}{n-1}\binom{\ell+2m-3}{\ell-2}q^{\ell}.
\end{equation}
We change the order of summation and shift the index to get
\begin{equation*}
F(z,q)\mkern-1mu=\mkern-1mu\sum_{\ell \ge 2}\mkern-1.5mu\binom{\ell+2m-3}{\ell-2}\mkern-.5mu q^{\ell} \mkern-4.1mu \sum_{n\ge  \ell-1}\mkern-1.5mu\binom{2n-\ell}{n-1}\mkern-.5mu  z^{n}
\mkern-1mu=\mkern-1mu q^2z\mkern-.5mu\sum_{\ell \ge 0}\mkern-1.6mu\binom{\ell+2m-1}{\ell}\mkern-.5mu  q^{\ell}z^{\ell}\mkern-.5mu\sum_{n\ge  0}\mkern-1.6mu\binom{2n+\ell}{n}\mkern-.5mu  z^{n}\mkern-1.6mu.
\end{equation*}
Introducing the Catalan generating function $C(z)=\frac{1-\sqrt{1-4z}}{2z}$, one has
\[
\sum_{n\ge  0}\binom{2n+\ell}{n} z^{n}   =  \frac{C(z)^\ell}{\sqrt{1-4z}}.
\]
While such a formula can be proven by convolution identities~\cite[Eq.~(5.72)]{GKP1994},
it is pleasant to give a bijective proof; we invite the reader to pause here and find it before reading on. 

The bijection consists in taking a walk of length $2n+\ell$ ending at altitude $\ell$, 
cutting it at the initial longest bridge, and after this, at the last passage at each altitude.  
This gives $\sum_{n\ge  0}\binom{2n+\ell}{n} z^{2n+\ell}   =  \frac{1}{\sqrt{1-4z}} (z C(z^2))^\ell$.

Going back to the quest for simplifying $F(z,q)$, we thus obtain
\begin{align}
F(z,q)&= \frac{q^2 z}{\sqrt{1-4z}} \cdot \frac{1}{\big(1-qzC(z)\big)^{2m}}, \label{eqFzq}
\end{align}
where in the last step we used the generating function identity $G(z)=\frac{1}{(1-z)^{2m}}=\sum_{\ell\ge 0}\binom{\ell+2m-1}{\ell}z^\ell$.
Note that Eq.~\eqref{eqFzq} suggests that there may be a bijective proof of Formula~\eqref{Kratproduct} 
using links with bridges and arches 
(instead of Krattenthaler's tour de force relying  on determinants and jeu de taquin). 

Now, in order to get the limit laws, our key observation is that Eq.~\eqref{eqFzq} is a composition scheme of shape  
$F(z,q)=M(z) G\big(q H(z)\big)$, where $H(z)=zC(z)$,
and the probability generating function (under the Gibbs measure) of the number of contacts is given by
\[
\E(v^{X_n(q)})= \frac{[z^{n-1} v^k]F(z,qv)}{[z^{n-1}]F(z,q)}.
\]
Next, we just apply singularity analysis with values $\lambdaG=-2m$, $\rhoG=1$, $\lambdaH=\frac12$. The scheme is critical for $q_c=2$,
where one gets the Mittag-Leffler distribution $\frac{1}{\sqrt{2}}\operatorname{ML}(\frac12,2m-\frac{1}{2})$,  which can be seen (from its moments) to be the same as the chi distribution $\chi(2m)$.  
This gives the theorem.
\end{proof}

\subsection{Returns to zero in coloured walks}
Let $m>0$  be an integer.
An $m$-coloured bridge is an $m$-tuple $(B_1,\dots,B_m)$ of (possibly empty) bridges $B_i$.
As a visual representation, we think of them appended one after the other, $B_i$ is coloured in colour $i$.
Note that not all colours need to be present. 
Let an $m$-coloured walk be an $m$-coloured bridge to which a final walk is appended that never returns to the $x$-axis.    
See Andrews~\cite{Andrews2007} and~\cite{GhoshDastidarWallner2024,HopkinsOuvry2021} for some combinatorial properties 
of these walks, and links with multicompositions.

We now prove that their number of returns to zero follows a $\chi(m)$ distribution. In particular, this gives the half-normal distribution for $m=1$ (which extends Theorem~\ref{the:DyckMotzkin} to unconstrained walks, see~\cite{Wallner20halfnormal}), 
the Rayleigh distribution for $m=2$, and the Maxwell distribution for $m=3$.  

\begin{figure}[hbt!]
\centering
\includegraphics[width=0.99\textwidth]{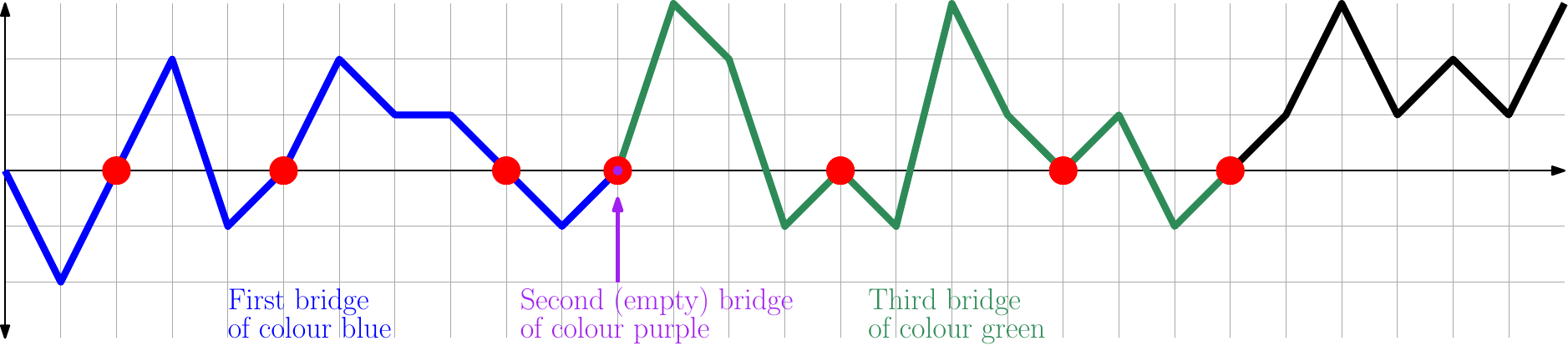}
\caption{A $3$-coloured walk (only $2$ of the available $3$ colours are present) with $7$ returns to zero (red dots).} 
\label{fig:returnsincolouredwalk}
\end{figure}
\vspace{-4mm}

\begin{theorem}[Phase transitions for returns to zero]
\label{the:DyckMotzkinWalks}
Under the Gibbs model, the distribution of the number $X_n(q)$ of returns to zero in $m$-coloured walks of length $n$ has, after rescaling for $n\rightarrow +\infty$, a phase transition at~$q_c=1$~and follows, either a negative binomial, chi, or~Gaussian~distribution. 
\bigskip

\centering\begin{tabular}{l@{\hskip 10mm }ccc}
\toprule
Parameter $q$ &$0<q<1$ & $q=1$ &  $q>1$\\
\midrule
Limit law of $X_n(q)$ & Negative binomial & Chi distribution & Gaussian \\
& $\nb(m,1-q)$ & $\chi(m)$ & $\mathcal{N}(0,1)$\\
\bottomrule
\end{tabular}
\end{theorem}

\begin{proof}[Proof (Sketch)]
We start by showing that returns to zero in $m$-coloured walks satisfy the composition scheme
$F(z,q) = M(z) G(qA(z))$.
First, let $B(z)$ be the generating function of bridges. 
Each bridge can be decomposed into a sequence of minimal bridges (or generalized arches), which have no return to zero between their extremities (yet jumps might cross the $x$-axis). 
Hence, $A(z) = 1-\frac{1}{B(z)}$.
Second, the $m$-tuple corresponding to an $m$-coloured bridge has the generating function $G(A(z))$, with $G(z)=\frac{1}{(1-z)^m}$.
Third, the final part of the walk (that never returns to the $x$-axis)
corresponds to $M(z)=W(z)/B(z)$, since each walk can be factored into an initial bridge and this final part.

Then, as before, the probability generating function (under the Gibbs measure) 
is 
\begin{equation*}
\E(v^{X_n(q)})= \frac{[z^{n} v^k]F(z,qv)}{[z^{n}]F(z,q)}.
\end{equation*}
We get the values $\lambdaG=-m$, $\rhoG=1$, and $\lambdaM=-1/2$. 
Furthermore, for any directed walk model~\cite{bafl02}, $B(z)$ has a critical exponent $\lambda_B=-1/2$, and therefore $\lambda_A = 1/2$.
This also means that $B(z)$ diverges at the singularity, and thus $A(z)$ is equal~to~$1$.
Hence, the scheme is universally critical at $q_c=1$, 
where  we get $\frac{1}{\sqrt{2}}\operatorname{ML}(\frac{1}{2},\frac{m-1}{2})$,  
which can be seen (from its moments) to be the same as the chi distribution $\chi(m)$. 
Other cases follow by singularity~analysis.
\end{proof}
\pagebreak

\section{Conclusion}
We unified the recent results of~\cite{BKW2021a} with the classical results of~\cite{Bender1973,FS2009}
to obtain phase transitions for Gibbs models under the umbrella of composition schemes. 
This allows us to obtain a variety of limit laws, leading to new results, as well as summarizing and generalizing classical results in 
analytic combinatorics. It also explains the universality hidden behind some phase transitions up to now sporadically observed in the literature.
In the full version of this article, we give several extensions of Theorem~\ref{the:compoMain}, 
thus treating many other examples of more general combinatorial constructions.

It is interesting to have an informal physicist look at our results: 
in statistical mechanics, the Gibbs measure can be seen as $q^k =\frac{\exp(-k/T)}{Z(1/T)}$, where $T$ is the temperature of the model.
Accordingly, $T\rightarrow 0$ (i.e., $q$ is very small) gives a frozen ``solid'' phase (typically leading to a discrete distribution),
while $T\rightarrow +\infty$ gives a ``gaseous'' phase (typically leading to a Gaussian distribution),
and, around (or at) a critical temperature $T_c$,  this gives a ``liquid'' phase (where the wild things are: one often observes at this location an unexpected fancy distribution).
Our article is one more illustration of this informal paradigm.

\ifthenelse{\equal{\compilationdestination}{arxiv}}
{\bibliographystyle{cyrbib}}
{\bibliographystyle{plainurl}}
\bibliography{p007-Banderier}

\begin{thebibliography}{10}

\bibitem{Aigner2021}
Florian Aigner.
\newblock \href{https://arxiv.org/abs/1810.08022}{A new determinant for the
  $q$-enumeration of alternating sign matrices}.
\newblock  \textit{Journal of Combinatorial Theory, Series A}, 180:27~pp.,
  2021.

\bibitem{Andrews2007}
George~E. Andrews.
\newblock \href{https://georgeandrews1.github.io/publications.html}{The theory
  of compositions. {IV}: {Multicompositions}}.
\newblock  \textit{The Mathematics Student}, 2007:25--31, 2007.

\bibitem{BD2015}
Cyril Banderier and Michael Drmota.
\newblock \href{http://lipn.univ-paris13.fr/~banderier/Papers/Alg.pdf}{Formulae
  and asymptotics for coefficients of algebraic functions}.
\newblock  \textit{Combinatorics, Probability and Computing}, 24(1):1--53,
  2015.

\bibitem{bafl02}
Cyril Banderier and Philippe Flajolet.
\newblock \href{http://dx.doi.org/10.1016/S0304-3975(02)00007-5}{Basic analytic
  combinatorics of directed lattice paths}.
\newblock  \textit{Theoretical Computer Science}, 281(1-2):37--80, 2002.

\bibitem{BFSS2001}
Cyril Banderier, Philippe Flajolet, Gilles Schaeffer, and Mich{\`e}le Soria.
\newblock \href{https://lipn.univ-paris13.fr/~banderier/Papers/rsa.pdf}{Random
  maps, coalescing saddles, singularity analysis, and {A}iry phenomena}.
\newblock  \textit{Random Structures \& Algorithms}, 19(3-4):194--246, 2001.

\bibitem{BKW2021a}
Cyril Banderier, Markus Kuba, and Michael Wallner.
\newblock \href{http://dx.doi.org/10.48550/arXiv.2103.03751}{Phase transitions
  of composition schemes: {M}ittag-{L}effler and mixed-{P}oisson
  distributions}.
\newblock  \textit{Annals of Applied probability}, 2024.
\newblock 53~pp., to appear.

\bibitem{BOR2019}
Nicholas~R. Beaton, A.~L. Owczarek, and Andrew Rechnitzer.
\newblock \href{http://dx.doi.org/10.37236/8024}{Exact solution of some quarter
  plane walks with interacting boundaries}.
\newblock  \textit{Electronic Journal of Combinatorics}, 26(3 (P53)), 2019.

\bibitem{Bender1973}
Edward~A. Bender.
\newblock \href{http://dx.doi.org/10.1016/0097-3165(73)90038-1}{Central and
  local limit theorems applied to asymptotic enumeration}.
\newblock  \textit{Journal of Combinatorial Theory, Series A}, 15:91--111,
  1973.

\bibitem{BGR2009}
Alexei Borodin, Vadim Gorim, and Eric~M. Rains.
\newblock \href{http://dx.doi.org/10.1007/s00029-010-0034-y}{$q$-distributions
  on boxed plane partitions}.
\newblock  \textit{Selecta Mathematica}, 16(4):731--789, 2010.

\bibitem{BousquetMelou06}
Mireille Bousquet-M\'elou.
\newblock \href{http://dx.doi.org/10.1088/1742-6596/42/1/005}{Three osculating
  walkers}.
\newblock  \textit{Journal of Physics: Conference Series}, 42:35--46, 2006.

\bibitem{BEO1998}
Richard Brak, John~W. Essam, and Aleksander~L. Owczarek.
\newblock \href{http://dx.doi.org/10.1023/B:JOSS.0000026731.35385.93}{New
  results for directed vesicles and chains near an attractive wall}.
\newblock  \textit{Journal of Statistical Physics}, 93:155--192, 1998.

\bibitem{BEO2001}
Richard Brak, John~W. Essam, and Aleksander~L. Owczarek.
\newblock \href{http://dx.doi.org/10.1023/A:1004819507352}{Scaling analysis for
  the adsorption transition in a watermelon network of $n$ directed
  non-intersecting walks}.
\newblock  \textit{Journal of Statistical Physics}, 102:997--1017, 2001.

\bibitem{ChelikavadaPanzo2023}
Aksheytha Chelikavada and Hugo Panzo.
\newblock \href{https://arxiv.org/abs/2311.04623}{Limit theorems for fixed
  point biased permutations avoiding a pattern of length three}.
\newblock  \textit{arXiv}, 2023.

\bibitem{CiucuKrattenthaler2002}
Mihai Ciucu and Christian Krattenthaler.
\newblock \href{http://dx.doi.org/10.1006/jcta.2002.3288}{Enumeration of
  lozenge tilings of hexagons with cut off corners}.
\newblock  \textit{Journal of Combinatorial Theory, Series A}, 100:201--231,
  2002.

\bibitem{DrmotaGimenezNoy2011}
Michael Drmota, Omer Gim{\'e}nez, and Marc Noy.
\newblock \href{http://dx.doi.org/10.1016/j.jcta.2011.04.010}{Degree
  distribution in random planar graphs}.
\newblock  \textit{Journal of Combinatorial Theory. Series A},
  118(7):2102--2130, 2011.

\bibitem{DrmotaSoria95}
Michael Drmota and Mich{\`e}le Soria.
\newblock \href{http://dx.doi.org/10.1016/0304-3975(94)00294-S}{Marking in
  combinatorial constructions: generating functions and limiting
  distributions}.
\newblock  \textit{Theoretical Computer Science}, 144(1-2):67--99, 1995.

\bibitem{DrmotaSoria97}
Michael Drmota and Mich{\`e}le Soria.
\newblock \href{http://dx.doi.org/10.1137/S0895480194268421}{Images and
  preimages in random mappings}.
\newblock  \textit{SIAM Journal on Discrete Mathematics}, 10(2):246--269, 1997.

\bibitem{Boltz1}
Philippe Duchon, Philippe Flajolet, Guy Louchard, and Gilles Schaeffer.
\newblock \href{http://dx.doi.org/10.1017/S0963548304006315}{Boltzmann samplers
  for the random generation of combinatorial structures}.
\newblock  \textit{{Combinatorics, Probability and Computing}},
  13(4-5):577--625, 2004.

\bibitem{Elizalde04}
Sergi Elizalde.
\newblock \href{http://dx.doi.org/10.37236/1804}{Multiple pattern avoidance
  with respect to fixed points and excedances}.
\newblock  \textit{Electronic Journal of Combinatorics}, 11(1):40~pp., 2004.

\bibitem{Feierl2012}
Thomas Feierl.
\newblock \href{http://dx.doi.org/10.1088/1751-8113/45/9/095003}{The height of
  watermelons with wall}.
\newblock  \textit{Journal of Physics A: Mathematical and Theoretical}, 45(9),
  2012.

\bibitem{Feierl2009}
Thomas Feierl.
\newblock \href{http://dx.doi.org/10.1016/j.ejc.2012.07.021}{The height and
  range of watermelons without wall}.
\newblock  \textit{European Journal of Combinatorics}, 34(1):138--154, 2013.

\bibitem{Feierl2014}
Thomas Feierl.
\newblock \href{http://dx.doi.org/10.1002/rsa.20467}{{Asymptotics for the
  number of walks in a Weyl chamber of type B}}.
\newblock  \textit{Random Structures \& Algorithms}, 45(2):261--305, 2014.

\bibitem{FillFlajoletKapur2005}
James~Allen Fill, Philippe Flajolet, and Nevin Kapur.
\newblock \href{http://dx.doi.org/10.1016/j.cam.2004.04.014}{Singularity
  analysis, {Hadamard} products, and tree recurrences}.
\newblock  \textit{Journal of Computational and Applied Mathematics},
  174(2):271--313, 2005.

\bibitem{Fisher1984}
Michael~E. Fisher.
\newblock \href{http://dx.doi.org/10.1007/BF01009436}{Walks, walls, wetting,
  and melting}.
\newblock  \textit{Journal of Statistical Physics}, 34:667--729, 1984.

\bibitem{FS2009}
Philippe Flajolet and Robert Sedgewick.
\newblock \href{http://algo.inria.fr/flajolet/Publications/book.pdf}{
  \textit{{Analytic Combinatorics}}}.
\newblock Cambridge University Press, 2009.

\bibitem{FlajoletSoria93}
Philippe Flajolet and Mich\`ele Soria.
\newblock \href{http://dx.doi.org/10.1016/0012-365X(93)90364-Y}{General
  combinatorial schemas: {G}aussian limit distributions and exponential tails}.
\newblock  \textit{Discrete Mathematics}, 114(1-3):159--180, 1993.

\bibitem{GhoshDastidarWallner2024}
Manosij Ghosh~Dastidar and Michael Wallner.
\newblock \href{https://arxiv.org/pdf/2402.17849}{Bijections and congruences
  involving lattice paths and integer compositions}.
\newblock  \textit{arXiv}, 2024.
\newblock A short version of this work was published in the proceedings of
  \href{https://cgi.cse.unsw.edu.au/~eptcs/paper.cgi?GASCom2024.22}{GASCOM
  2024}.

\bibitem{Gourdon98}
Xavier Gourdon.
\newblock \href{http://dx.doi.org/10.1016/S0012-365X(97)00115-5}{Largest
  component in random combinatorial structures}.
\newblock  \textit{Discrete Mathematics}, 180(1-3):185--209, 1998.

\bibitem{GKP1994}
Ronald~L. Graham, Donald~E. Knuth, and Oren Patashnik.
\newblock \href{https://www-cs-faculty.stanford.edu/~knuth/gkp.html}{
  \textit{{Concrete Mathematics}}}.
\newblock Addison-Wesley Publishing Company, second edition, 1994.

\bibitem{HopkinsOuvry2021}
Brian Hopkins and St{\'e}phane Ouvry.
\newblock \href{http://dx.doi.org/10.1007/978-3-030-67996-5_16}{Combinatorics
  of multicompositions}.
\newblock In  \textit{Combinatorial and additive number theory IV}, pages
  307--321. Springer, 2021.

\bibitem{HopkinsLazarLinusson2021}
Sam Hopkins, Alexander Lazar, and Svante Linusson.
\newblock \href{http://dx.doi.org/10.1016/j.ejc.2023.103760}{On the
  $q$-enumeration of barely set-valued tableaux and plane partitions}.
\newblock  \textit{European Journal of Combinatorics}, 2021.

\bibitem{Hwang1998}
Hsien-Kuei Hwang.
\newblock \href{http://dx.doi.org/10.1006/eujc.1997.0179}{{On convergence rates
  in the central limit theorems for combinatorial structures}}.
\newblock  \textit{European Journal of Combinatorics}, 19:329--343, 1998.

\bibitem{Krattenthaler2006}
Christian Krattenthaler.
\newblock \href{http://dx.doi.org/10.1088/1742-6596/42/1/017}{Watermelon
  configurations with wall interaction: exact and asymptotic results}.
\newblock  \textit{Journal of Physics: Conference Series}, 42:179--212, 2006.

\bibitem{KrattenthalerGuttmannViennot2000}
Christian Krattenthaler, Anthony~J. Guttmann, and Xavier~G. Viennot.
\newblock \href{http://dx.doi.org/10.1088/0305-4470/33/48/318}{Vicious walkers,
  friendly walkers and {Y}oung tableaux {II}: with a wall}.
\newblock  \textit{Journal of Physics A: Mathematical and General},
  33:8835--8866, 2000.

\bibitem{Kropf2017}
Sara Kropf and Stephan Wagner.
\newblock \href{http://dx.doi.org/10.37236/6373}{On {$q$}-quasiadditive and
  {$q$}-quasimultiplicative functions}.
\newblock  \textit{Electronic Journal of Combinatorics}, 24(1):Paper No. 1.60,
  22 pages, 2017.

\bibitem{Mallows1957}
Colin~L. Mallows.
\newblock \href{http://dx.doi.org/10.2307/2333244}{{Non-null ranking models
  I}}.
\newblock  \textit{Biometrika}, 44:114--130, 1957.

\bibitem{OP2019}
Aleksander~L. Owczarek and Thomas Prellberg.
\newblock \href{http://dx.doi.org/10.1063/1.5083149}{Exact solution of pulled,
  directed vesicles with sticky walls in two dimensions}.
\newblock  \textit{Journal of Mathematical Physics}, 60:8~pp., 2019.

\bibitem{Roitner2020b}
Valerie Roitner.
\newblock
  \href{http://bica.the-ica.org/Volumes/89/Reprints/BICA2020-01-Reprint.pdf}{Contacts
  and returns in 2-watermelons without wall}.
\newblock  \textit{Bulletin of the Institute of Combinatorics and its
  Applications}, 89, 2020.

\bibitem{Stufler2022}
Benedikt Stufler.
\newblock \href{https://arxiv.org/abs/2204.06982}{Gibbs partitions: a
  comprehensive phase diagram}.
\newblock  \textit{Annales de l'Institut Henri Poincar\'e - Probabilit\'es et
  Statistiques}, 2023.
\newblock To appear.

\bibitem{TOR2014}
Rami Tabbara, Aleksander~L. Owczarek, and Andrew Rechnitzer.
\newblock \href{http://dx.doi.org/10.1088/1751-8113/47/1/015202}{{An exact
  solution of two friendly interacting directed walks near a sticky wall}}.
\newblock  \textit{Journal of Physics A: Mathematical and Theoretical}, 47(1),
  2014.

\bibitem{Vella2003}
Antoine Vella.
\newblock \href{http://dx.doi.org/10.37236/1690}{Pattern avoidance in
  permutations: {Linear} and cyclic orders}.
\newblock  \textit{Electronic Journal of Combinatorics}, 9(2):43~pp., 2003.

\bibitem{Wallner20halfnormal}
Michael Wallner.
\newblock \href{http://dx.doi.org/10.1016/j.ejc.2020.103138}{A half-normal
  distribution scheme for generating functions}.
\newblock  \textit{European Journal of Combinatorics}, 87:103138, 2020.

\bibitem{Wu2022}
Shuang Wu.
\newblock
  \href{http://dx.doi.org/10.1088/1751-8121/ac93cd}{{Sachdev--Ye--Kitaev model
  with an extra diagonal perturbation: phase transition in the eigenvalue
  spectrum}}.
\newblock  \textit{Journal of Physics A: Mathematical and Theoretical},
  55(41):415207, 2022.

\end{thebibliography}

\ifthenelse{\equal{\compilationdestination}{arxiv}}
{First version: arXiv, November 28, 2023.\\
This final version from July 17, 2024 corresponds to the article published by LIPIcs (with a different style for the bibliography).
}
{}
\end{document}